\documentclass[hidelinks,onefignum,onetabnum]{siamart220329} 
\usepackage{geometry}
\usepackage{comment}
\usepackage{amssymb}
\usepackage{enumerate}
\usepackage{color}
\usepackage{MnSymbol}
\usepackage{tikz}
\usepackage{float}

\usepackage{lipsum}
\usepackage{amsfonts}
\usepackage{graphicx}
\usepackage{epstopdf}
\usepackage{algorithmic}
\ifpdf
  \DeclareGraphicsExtensions{.eps,.pdf,.png,.jpg}
\else
  \DeclareGraphicsExtensions{.eps}
\fi

\newsiamremark{remark}{Remark}
\newsiamremark{hypothesis}{Hypothesis}
\crefname{hypothesis}{Hypothesis}{Hypotheses}
\newsiamthm{claim}{Claim}

\headers{Tree tensor network operators for long-range interactions}{G. Ceruti, D. Kressner, D. Sulz}

\title{Low-rank tree tensor network operators \\ for long-range pairwise interactions\thanks{{\bf Funding:} The work of Dominik Sulz was funded by the Deutsche Forschungsgemeinschaft (DFG, German Research Foundation) through the Research Unit FOR 5413/1, Grant No. 465199066 and the Research Unit TRR 352, Project-ID 470903074. The work of Gianluca Ceruti was partly funded by the SNSF research project \emph{Fast algorithms from low-rank updates},
grant number: 200020\_178806.}}

\author{Gianluca Ceruti\thanks{University of Innsbruck, Austria. (\email{gianluca.ceruti@uibk.ac.at}).} 
\and Daniel Kressner\thanks{ANCHP, EPF Lausanne, Switzerland. (\email{daniel.kressner@epfl.ch}).} \and Dominik Sulz\thanks{University of Tuebingen, Germany. 
  (\email{dominik.sulz@uni-tuebingen.de}
  ).}}

\usepackage{amsopn}

\ifpdf
\hypersetup{
  pdftitle={TTNO},
  pdfauthor={G. Ceruti, D. Kressner, D. Sulz}
}
\fi

\renewtheorem{remark}[theorem]{Remark}

\newcommand{\norm}[1]{\left\lVert#1\right\rVert}

\newcommand{\C}{\mathbb{C}}

\newcommand{\U}{{\bf{U}}}
\newcommand{\I}{{\bf{I}}}
\newcommand{\A}{{\bf{A}}}
\newcommand{\V}{{\boldsymbol{V}}}
\newcommand{\bS}{{\boldsymbol{S}}}\newcommand{\bR}{{\boldsymbol{R}}}

\newcommand{\ba}{{\bf{a}}}
\newcommand{\be}{{\bf{e}}}
\newcommand{\bh}{{\bf{h}}}
\newcommand{\bu}{{\bf{u}}}
\newcommand{\boldbeta}{\boldsymbol{\beta}}

\newcommand{\vect}{\mathop{\mathrm{vec}}}
\newcommand{\rank}{\mathop{\mathrm{rank}}}

\DeclareMathOperator{\mat}{\bf{Mat}}

\begin{document}
\maketitle

\begin{abstract}
Compactly representing and efficently applying linear operators are fundamental ingredients in tensor network methods for simulating quantum many-body problems and solving high-dimensional problems in scientific computing. In this work, we study such representations for tree tensor networks, the so called  tree tensor network operators (TTNOs), paying particular attention to Hamiltonian operators that involve long-range pairwise interactions between particles. Generalizing the work by Lin, Tong, and others on matrix product operators, we establish a direct connection between the hierarchical low-rank structure of the interaction matrix and the TTNO  property. This connection allows us to arrive at very compact TTNO representations by compressing the interaction matrix into a hierarchically semi-separable matrix.
Numerical experiments for different quantum spin systems validate our results and highlight the potential advantages of TTNOs over matrix product operators.
\end{abstract}



\section{Introduction}

In a wide variety of situations, tensor network methods have demonstrated their power to address the computational challenges imposed by high dimensionality. Most notably, this includes matrix product states (MPS) for many-body quantum systems~\cite{Schollwock2011} or the (mathematically equivalent) tensor train (TT) decomposition~\cite{Os11} for applications in engineering and scientific computing; see~\cite{surveyBachmayr2023,surveyGKT13,Khoromskij2018} for surveys. Many of these approaches crucially rely on a suitable representation of the operator describing the system. In this work, we will consider operators that describe pairwise interactions between sites (or modes) of a system:
\begin{align} \label{eq:Ham_intro}
    \mathcal H: \C^{n_1\times \cdots \times n_d} \to \C^{n_1\times \cdots \times n_d}, \quad \mathcal H = \sum_{k=1}^d \mathcal{D}^{(k)} + \sum_{1\le i < j \le d} \beta(i,j)  \mathcal{A}^{(i)} \mathcal{A}^{(j)},
\end{align}
where ${\mathcal A^{(i)}}$ and ${\mathcal D^{(i)}}$ represent operators that act on the $i$th site only, while the coefficients $\beta(i,j)$ characterize the strength of interactions between pairs of sites. The matrix representation of a single-site operator like ${\mathcal A^{(i)}}$ takes the form $\I_{n_d} \otimes \dots  \otimes \I_{n_{i+1}} \otimes \A_i \otimes \I_{n_{i-1}} \otimes\dots  \otimes \I_{n_{1}}$ for some matrix $\A_i \in \mathbb C^{n_k\times n_k}$, where $\I$ denotes an identity matrix of suitable size and $\otimes$ denotes the usual Kronecker product.
In scientific computing, an operator of the form~\eqref{eq:Ham_intro} arise from the discretization of partial differential equations (PDEs) on $d$-dimensional hypercubes, while for quantum systems a Hamiltonian of the form~\eqref{eq:Ham_intro} describes the interaction between pairs of particles. If the matrix $\boldsymbol{\beta} \in \mathbb C^{d\times d}$ containing the coefficients $\beta(i,j)$ is sparse, only a few pairs interact with each other. In particular, when $\boldsymbol{\beta}$ is banded, only short-range interactions are allowed, that is, only nearby sites interact with each other. When $\boldsymbol{\beta}$ does not have banded structure, distant sites interact with each other. The presence of such long-range interactions significantly complicates the use of tensor network methods, including the compact representation of $\mathcal H$.

%

To efficiently use MPS / TT representations for, e.g., computing ground states or performing time evolution it is advantageous to avoid the canonical representation~\eqref{eq:Ham_intro} of the operator $\mathcal H$ and use a matrix product operator (MPO) representation~\cite{PMCV2012,Schollwock2011,VG2004}. For example, when applying $\mathcal H$ in canonical representation to a tensor in TT decomposition, the associated TT representation ranks (which critically determine memory complexity) grow by a factor around $d^2/2$. If, instead, $\mathcal H$ is represented in  MPO format, its potentially much smaller MPO ranks determine this growth. In particular, it is well known that the MPO ranks are constant (not depending on $d$) for short-range interactions. In the context of PDEs, such MPO representations for $\mathcal H$ have been constructed in~\cite{Dolgov2013,Kazeev2013}, highlighting the connection between the MPO representation and the quasi-separability~\cite{Eidelman} of~$\boldsymbol{\beta}$. In the context of quantum systems, a similar connection has been made in~\cite{LT21}, additionally discussing the compression of~$\boldsymbol{\beta}$ as a quasi-separable matrix. This is particularly important for long-range interactions like the Coulomb interaction $\beta(i,j) = 1/|i-j|$, for which $\boldsymbol{\beta}$ is not a quasi-separable matrix but can be well approximated by one. A quasi-optimal algorithm for quasi-separable approximation is described in~\cite{Massei2018}, but to the best of our knowledge no software implementation is publicly available.

While MPS excels at compressing ground states for short-range and/or translation-invariant interactions~\cite{MR2338267}, it may struggle to attain compact representations in the presence of long-range interactions. Due to the frequent occurrence of long-range interactions in applications, this motivates the need to consider more general tensor network formats. In this work, we will consider Tree Tensor Networks (TTNs) at a level of generality that includes any connected loop-free tensor network. TTNs include MPS / TT as a special case when choosing a highly unbalanced, degenerate binary tree. However, TTNs are more commonly used with a balanced (binary) tree, corresponding to the hierarchical Tucker format~\cite{G10} in scientific computing.
This comes with the advantage that the distances between nodes within the tree scale logarithmically with $d$. Consequently, the representation rank in a TTN can be expected to remain smaller than for MPS, where this distance scales linearly.
%
In the realm of quantum systems, TTNs have demonstrated promising numerical accuracy in capturing long-range interactions and correlations among particles~\cite{KRB2020,sulz2023}. However, the efficient realization of this promise requires a compact representation of the operator $\mathcal H$ that corresponds to the TTN format. 
The construction of such Tree Tensor Network Operators (TTNOs) has been discussed for nearest-neighbor interaction in~\cite{Tobler2012}.
In \cite{milbradt2023TTNO}, state diagrams are used for the construction of quantum Hamiltonians in the TTN format.

In this paper, we generalize the work by Lin and Tong~\cite{LT21} on MPOs to TTNOs for general tree tensor networks. Our main contribution is show a direct connection between the hierarchical semi-separable (HSS) decomposition~\cite{XXCB14} of the interaction matrix $\boldsymbol{\beta}$ and the TTNO for $\mathcal H$ from~\eqref{eq:Ham_intro}. This connection allows to leverage existing algorithms and software, such as the hm-toolbox~\cite{MRK20}, for compressing and working with HSS matrices.
We prove that when the interaction matrix $\boldsymbol{\beta}$ has HSS rank of $k$ then the maximum rank of the TTNO is bounded by $k+2$ and we provide an explicit construction of this TTNO. We also quantify the error (in the spectral norm) inflicted on $\mathcal H$ by the compression of $\boldsymbol{\beta}$, which predicts that the error will remain small for interaction matrices typically encountered in practice, including Coulomb interactions. This is confirmed by our numerical experiments, which exhibit that small HSS ranks are commonly encountered in relevant quantum scenarios, without needing to impose strong assumptions on the interaction matrix, such as decaying interactions or translational invariance. 

The rest of this paper is structured as follows. In Section~\ref{sec:ttn}, we recall TTNs and TTNOs. In Section~\ref{sec:ttnounstruc}, we first show how to construct TTNO without imposing conditions on $\boldsymbol{\beta}$. Section~\ref{sec:ttnohss} contains our main result, the construction of a TTNO from an HSS decomposition of $\boldsymbol{\beta}$, including an error bound when $\boldsymbol{\beta}$ is compressed in the HSS format. In Section~\ref{sec:numexp}, we verify our results numerically for several quantum spin systems. Further, we compare the TTNO representation for a balanced binary trees with the TTNO for a degenerate binary tree (corresponding to an MPO) in the context of long-range interactions.  

\section{Tree tensor networks} \label{sec:ttn}

Following~\cite{CeLW21,CLS23}, this section introduces a tree tensor network formalism that includes the MPS/TT and HT tensor formats discussed in the introduction. 
%

\subsection{Dimension tree}

Tree tensor networks rely on a recursive partition of the dimensions $1$ through $d$. This partitioning is represented by a \emph{dimension tree} $\bar \tau$. In the following definition,
$L(\tau)$ denotes the set of leaves of a tree $\tau$ and we always assume that it consists of a (nonempty) set of consecutive integers.

\begin{definition}[dimension tree] \label{def:dimtree}
For $d\ge 1$, a \emph{dimension tree} $\bar \tau$ with leaves $L(\bar \tau) = \{1,\ldots, d\}$ is recursively defined as follows:
\begin{enumerate}[(i)] 
 \item For a dimension tree $\tau$ with more than one leaf, $L(\tau)$ is partitioned as 
\[
  L(\tau) = L(\tau_1) \mathbin{\dot{\cup}} L(\tau_2) \mathbin{\dot{\cup}} \cdots \mathbin{\dot{\cup}} L(\tau_m)
\]
for some $m \ge 2$. 
 Corresponding to this partition, the root of $\tau$ has $m$ children that are roots of
 dimension trees $\tau_1, \dots, \tau_m$ such that
 $L(\tau_i) = \mathcal L_i$ and
\begin{equation} \label{eq:partitionrel}
    \max L(\tau_i) < \min L(\tau_j) \qquad \forall i < j.
\end{equation}
We will write $\tau = (\tau_1, \dots, \tau_m)$.
 \item For a singleton $L(\tau) = \{\ell\}$, it holds that $\tau = \ell$. 
 
\end{enumerate}
\end{definition}
Any tree $\tau = (\tau_1, \dots, \tau_m)$ appearing in the recursive construction of Definition~\ref{def:dimtree} is a subtree of $\bar \tau$. The set of all such subtrees will be denoted by $\mathcal T(\bar \tau)$. 
%
%
%
By condition \eqref{eq:partitionrel}, 
for any subtree $\tau = (\tau_1 ,\dots  , \tau_m ) \in \mathcal T(\bar \tau)$ there exist integers $\ell_1 < \ell_2 < \cdots < \ell_{m+1}$ such that
\[
 L(\tau_1) = \{ \ell_1, \ell_1 + 1, \ldots, \ell_2-1 \},\quad \ldots,\quad   L(\tau_m) = \{ \ell_m, \ell_m + 1, \ldots, \ell_{m+1}-1 \}.
\]
For instance, when $d = 6$, an example of a dimension tree is $\bar \tau = ( (1,2,3),(4,5,6) )$ -- see Figure \ref{fig:tree}.
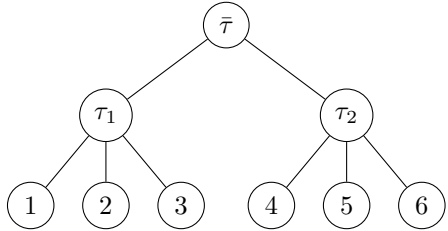
\begin{figure}[H]
	\label{fig:tree} 
	\begin{center}	
	\begin{tikzpicture}[every node/.style={},level 1/.style={sibling distance=32mm},level 2/.style={sibling distance=10mm},level distance = 12mm]
		\node[circle,draw] { $\bar{\tau}$ }
		child { node[circle,draw] { $\tau_1$ }
			child { node[circle,draw] {1} }
			child { node[circle,draw] {2} }
			child { node[circle,draw] {3} } }
		child { node[circle,draw] { $\tau_2$ }
			child { node[circle,draw] {4} }
			child { node[circle,draw] {5} }
			child { node[circle,draw] {6} } };
		
	\end{tikzpicture}
	
	  \caption{Graphical representation of the tree $\bar{\tau}=(\tau_1,\tau_2)$ with $\tau_1 = (1,2,3)$ and $\tau_2 = (4,5,6)$.}
	\end{center}	
	
\end{figure} 

\subsection{Definition of tree tensor networks}

To work with tensors, it is useful to begin by recalling the concept of matricizations. 
Consider a tensor $C \in \C^{n_1 \times \dots \times n_d}$ and let $I$ be any subset of $\{1,\dots,d\}$. The $I$-matricization of $C$ is denoted as $\mat_I(C) \in \C^{n_I \times n_{\neg I}}$, where $n_i = \prod_{i \in I} n_{\tau_i}$ and $n_{\neg I} = \prod_{i \not\in I} n_{\tau_i}$, and it is obtained by merging the indices of the modes in $I$ into row indices while the remaining ones are merged into column indices, following the reverse lexicographic order.

With each leaf $\ell \in L(\bar \tau) = \{1,\ldots,d\}$ of a dimension tree $\bar \tau$, we associate a \emph{basis matrix} $\U_\ell \in \C^{n_\ell \times r_\ell}$. With each subtree $\tau = (\tau_1 ,\dots  , \tau_m ) \in \mathcal T(\bar \tau)$, we associate a {\it transfer tensor} $C_{\tau} \in \C^{r_{\tau_1} \times \dots \times r_{\tau_m} \times r_\tau}$. The integers $r_\tau$ determine the sizes of the basis matrices and transfer tensors. We always set $r_{\bar \tau} = 1$ and neglect the trailing singleton dimension of the transfer tensor $C_{\bar \tau}$ at the root. In particular, if $\bar \tau = (\tau_1,\tau_2)$ then $C_{\bar \tau}$ becomes an $r_{\tau_1}\times r_{\tau_2}$ matrix. We are now in the position to define a tree tensor network recursively~\cite{CLS23,CeLW21}.
\begin{definition}[Tree tensor network] \label{def:ttn}
For a given dimension tree $\bar \tau$, consider basis matrices $\U_\ell$ and transfer tensors $C_\tau$ defined as above. 
Then the corresponding \emph{tree tensor network} takes the following form: 
	\begin{enumerate}[(i)]
		\item
		For each leaf  $\ell \in L(\bar \tau) = \{1,\ldots,d\}$, we set 
		$$X_\ell := \U_\ell \in \C^{n_\ell \times r_\ell}.  $$
		\item
		For each subtree $\tau = ( \tau_1, \dots, \tau_m) \in \mathcal T(\bar\tau)$,  
		we set 
		$n_\tau := \prod_{i=1}^m  n_{\tau_i}$ and
		\begin{align*}
		& X_\tau := C_\tau \bigtimes_{i=1}^m \U_{\tau_i} 
		\in \mathbb{C}^{n_{\tau_1} \times \dots \times n_{\tau_m} \times r_\tau},
		\\
		& \U_{\tau} := \mat_{1,\ldots,m}( X_\tau ) \in \mathbb{C}^{n_\tau \times r_\tau} \ .
		\end{align*}
	\end{enumerate}
\end{definition}

At the root $\bar \tau = ( \tau_1, \dots, \tau_m)$, we arrive at a tensor of the form $X_{\bar \tau} := C_{\bar \tau} \bigtimes_{i=1}^m \U_{\tau_i} \in \mathbb C^{n_{\tau_1}\times \cdots \times n_{\tau_m}}$, which is a reshape of the $d$th order tensor $X\in \mathbb{C}^{n_1\times \cdots \times n_d}$ represented by the entire tree tensor network. In particular, $U_{\bar \tau} \in \mathbb{C}^{n_1\cdots n_d}$ is the vectorization of both, $X_{\bar \tau}$ and $X$.
Tensor trains/matrix product states are tree tensor networks of maximal height, while tensors in Tucker format have height~$1$. Furthermore, we define the \emph{representation rank} of a tree tensor network as the largest $r_\tau$ for $\tau \in \mathcal T(\bar{\tau}) \cup L(\bar \tau)$, i.e., as the maximal representation rank appearing in the TTN.

Definition~\ref{def:ttn} recursively defines tree tensor networks via the Tucker format. For later purposes, it is helpful to recall the unfolding formula for the Tucker format from~\cite{KoB09}:
\begin{align}
    \mat_j \Big( C \bigtimes_{i=1}^m \U_{\tau_i} \Big) = \U_j \mat_j (C) \bigotimes_{i \neq j} \U_{\tau_i}^\top.\label{eq:unfolding_formula}
\end{align}
In the mathematical literature, tree tensor networks for binary trees have been studied as hierarchical tensors \cite{Ha12}, tensors in hierarchical
Tucker format \cite{KT14,G10} and with general trees as tensors in tree-based tensor format \cite{FaHN15,FaHN18}. In quantum chemistry, tree tensor networks are employed -- among others -- for the multilayer multiconfiguration time-dependent Hartree method (ML-MCTDH)~\cite{Manthe2008,VM2011}.

The following result relates the ranks of a tree tensor network with the ranks of its matricizations; it can be found in, e.g.,~\cite[Theorem 11.12 and Lemma 11.15]{H19}.
\begin{theorem} \label{thm:ttnbounds}
For a given dimension tree $\bar \tau$, a tensor $X\in \mathbb{C}^{n_1\times \cdots \times n_d}$ admits a tree tensor network representation with
\[
 r_\tau = \mathrm{rank}\big( \mat_{L(\tau)}(X) \big), \quad \forall \tau \in T(\bar \tau) \cup L(\bar \tau).
\]

\end{theorem}

\subsection{Tree tensor network operators (TTNO)} \label{sec:ttno}
Having established Definition~\ref{def:ttn} of a tree tensor network, we can now introduce the concept of a tree tensor network operator (TTNO).

\begin{definition}[TTNO] 
\label{def:TTNO}
    Let $\mathcal H: \mathbb C^{n_1\times \cdots \times n_d} \to \mathbb C^{n_1\times \cdots \times n_d}$ be a linear operator, and $\widehat H \in \mathbb C^{(n_1 \dots n_d) \times (n_1 \dots  n_d)}$ its associated linear map, i.e.,
    $ \vect(\mathcal H(X))= \widehat H \vect(X)$ for all $X \in \mathbb C^{n_1 \times \dots \times n_d}$.
   We define $\mathcal{H}$ as a tree tensor network operator (TTNO) of representation rank $r$ on a given tree  $\bar \tau$ if the reshape $H \in \mathbb{C}^{n_1^2 \times \cdots \times n_d^2}$ of $\widehat{H}$ forms a tree tensor network of representation rank $r$ on $\bar{\tau}$. The tree tensor network $H$ is the TTNO representation of $\mathcal{H}$.
\end{definition}

Definition~\ref{def:TTNO} imposes a tree tensor network structure on a specific reshape of the linear map associated with the operator. To illustrate this construction, we first consider one of the simplest examples for trees of height $1$, that is, determining a TTNO that aligns with the Tucker format. Consider a linear operator of the form:
\begin{equation} \label{eq:rank1operator}
     \mathcal{H}(X) =  X \bigtimes_{j=1}^d \A_j 
    \quad \text{with} \quad
    \A_j \in \mathbb C^{n_j \times n_j} 
    \, .
\end{equation}
Its matrix representation $\widehat H$ is retrieved by applying the unfolding formula (\ref{eq:unfolding_formula}),
\[
    \vect(\mathcal{H}( X) ) 
    = \mat_{1,\ldots,m}(\mathcal{H}(X)) = \underbrace{\left( \A_d \otimes \cdots \otimes \A_1 \right)}_{=\hat H} \vect(X)  \, .
\]
Following well-established techniques~\cite{KT14,Schollwock2011} that leverage matrix/tensor low-rank approximation for approximating operators, the Tucker operator representation of $H$ is obtained by reshaping $\text{vec}(\widehat{H})$ into an $n_1^2 \times \dots \times n_d^2$ tensor, resulting in the outer product representation
$$H = \vect(\A_1) \circ \dots \circ \vect(\A_d) \ \in \ \mathbb C^{n_1^2 \times \dots \times n_d^2}.$$
This corresponds to a rank-$1$ Tucker operator
$H = C \bigtimes_{i=1}^d \vect(\A_i) \in \mathbb C^{n_1^2 \times \dots \times n_d^2} $ with $C=1$. 
The described construction easily extends to a general dimension tree $\bar \tau$, by letting each basis matrix contain $\vect(\A_1)$ and setting each transfer tensor to $1$. One can also accommodate a sum of $r$ operators of the form~\eqref{eq:rank1operator}, using that the sum of $r$ rank-$1$ TTNOs is a rank-$r$ TTNO, but the obtained representation may not be optimal in terms of representation ranks and storage complexity.
%

To conclude, we recall~\cite[\S 8]{KT14} how to efficiently apply a linear operator $\mathcal{H}$ in TTNO representation $H$ to a tree tensor network~$X$ with the same dimension tree $\bar \tau$. The evaluation of $\mathcal{H}(X)$ results again in a tree tensor network (of larger rank) and it can be efficiently computed as follows:
\begin{enumerate}
	\item[(i)] For each leaf $\tau = \ell$, we set
	\begin{align*}
		\left( \mathcal{H}(X) \right)_l = \left[ \A_1^{\ell} \U_\ell, \dots, \A_s^\ell \U_\ell \right],
	\end{align*}
	where $\mathbf{A}_j^\ell \in \mathbb C^{n_\ell \times n_\ell}$ is the matrix obtained from reshaping the $j$th column of the basis matrix at the $\ell$th leaf of the TTNO $H$.

	\item[(ii)] For each subtree $\tau \in \mathcal T(\bar \tau)$, we set 
	\begin{align} \label{eq:kronTen}
		C_{\tau}^{\mathcal{H}(X)} = C_{\tau}^{H} \otimes C_{\tau}^{X}.
	\end{align}
        where $C_{\tau}^{H}$ denotes the transfer tensor of the TTNO representation $H$ at $\tau$, and $C_{\tau}^X$ denotes the transfer tensor of the TTN $X$ at $\tau$.
\end{enumerate}
The operation~\eqref{eq:kronTen} involves the Kronecker product of tensors. We recall that 
the tensor Kronecker product $C \in \C^{(n_1m_1)\times \dots \times (n_dm_d)}$ of tensors $A \in \C^{n_1\times \dots \times n_d}$ and $B \in \C^{m_1\times \dots \times m_d}$ is defined element-wise for $1\leq i_k \leq n_k$ and $1 \leq j_k \leq m_k$ via the relation
\begin{align*}
	C(j_1 + (i_1-1)m_1,\dots,j_d + (i_d-1)m_d) := A(i_1,\dots,i_d) B(j_1,\dots,j_d) \, .
\end{align*}
It is worth noting that the operations (i) and (ii) in the application of a TTNO are independent and can be performed in parallel. Moreover, the operation (ii) multiplies each rank of the original tensor network by the corresponding rank of the TTNO. Thus, a compact TTNO representation $H$, with small representation ranks, is of utmost interest.




\section{Construction of TTNOs for Hamiltonians} \label{sec:ttnounstruc}

We now focus on the Hamiltonian operator \eqref{eq:Ham_intro}, which consists of two parts. The first part, resembling a Laplacian operator, acts on individual sites. The second part describes interactions between pairs of sites. Because the treatment of the first part is well-established~\cite{KT14}, we focus on the second part, which poses additional challenges. 

To be specific, we investigate the TTNO representation of Hamiltonians having a matrix representation of the form
\begin{equation} \label{eq:hamiltonian}
	\widehat H = \sum_{1\le i < j \le d} \beta(i,j) \cdot \A^{(i)} \A^{(j)} \in \mathbb C^{(n_1\cdots n_d) \times (n_1\cdots n_d)},
\end{equation}
with $\A^{(k)} = \I_{n_d} \otimes \dots  \otimes \I_{n_{k+1}} \otimes \A_k \otimes \I_{n_{k-1}} \otimes\dots  \otimes \I_{n_{1}}$ and $\A_k \in \mathbb C^{n_k\times n_k}$ for $k  = 1, \dots, d$. The strictly upper triangular interaction matrix $\boldsymbol{\beta} = (\beta(i,j))_{i,j=1}^{d}$ captures the strength of interaction between pairs of sites. Often, $\beta(i,j) = f(|i-j|)$ for $i < j$ and a given function $f(\cdot)$;
the entries in the lower triangular part of $\boldsymbol{\beta}$ are set to zero.

To simplify the presentation, we primarily consider  binary dimension trees $\bar \tau$, i.e., each subtree of $\bar \tau$ takes the form
$\tau = (\tau_1,\tau_2)$. In Section~\ref{sec:generaltrees}, we will briefly discuss the extension 
to general trees. 

\subsection{Construction of TTNOs: Unstructured case} \label{subsec:unstruct_case}

The most straightforward approach to derive a TTNO representation of \eqref{eq:hamiltonian} follows from the observation that each summand has rank one. Summing up the TTNOs for each term thus results in a TTNO with representation rank $d(d-1)/2$; see Section~\ref{sec:ttno}. This rank is far from optimal; indeed, we will show that it can always be reduced to $\mathcal O(d)$, without imposing any assumption on the interaction matrix $\boldsymbol{\beta}$. 

To obtain a TTNO representation $H \in \mathbb C^{n^2_1 \times \cdots \times n^2_d}$ for~\eqref{eq:hamiltonian}, we start by vectorizing the summands:
\[
 \bh := \vect(H) = \sum_{1\le i < j \le d} \beta(i,j) \cdot \ba_{}^{(i,j)} \in \mathbb C^{n^2_1\cdots n^2_d}.
\]
where
\begin{align} 
\ba^{(i,j)} & := \be_{d} \otimes \cdots  \otimes \be_{{j+1}} \otimes \vect(\A_j) \otimes \be_{{j-1}} \otimes\cdots  \otimes  \be_{{i+1}} \otimes \vect(\A_i) \otimes \be_{{i-1}} \otimes\cdots  \otimes \be_{{1}}, \label{eq:aijsanstau}
\end{align}
and $\be_k = \vect(\I_{n_k})$ denotes the vectorization of the identity matrix. More generally, given an arbitrary  tree $\tau$  and $i,j \in L(\tau)$ with $i< j$, we define the one- and two-site matrix representations as follows:
\begin{align} 
\label{eq:defai} 
\ba_{\tau}^{(i)} &= \bigotimes_{\ell \in L(\tau) \atop \ell > i} \be_{\ell}  \otimes \vect(\A_i)  \bigotimes_{\ell \in L(\tau) \atop \ell < i} \be_{\ell},
\\
\label{eq:defaij}
  \ba_{\tau}^{(i,j)} &= \bigotimes_{\ell \in L(\tau) \atop \ell > j} \be_{\ell}  \otimes \vect(\A_j) \bigotimes_{\ell \in L(\tau) \atop i < \ell < j}^{i+1} \be_{\ell} \otimes \vect(\A_i)  \bigotimes_{\ell \in L(\tau) \atop \ell < i} \be_{\ell},
\end{align}
with the understanding that the Kronecker products are executed in decreasing order with respect to $\ell$. Note that $\ba_{\tau}^{(i,j)}$ matches~\eqref{eq:aijsanstau} for $\tau = \bar \tau$.

Given a subtree $\tau \in \mathcal T(\bar\tau)$, let us consider the part of the Hamiltonian~\eqref{eq:hamiltonian} that contains all interactions between sites within $L(\tau)$. Following the discussion above, the (vectorized) Hamiltonian for $\tau$ takes the form
\begin{equation} \label{eq:htau}
 \bh_\tau := \sum_{i < j \atop i,j\in L(\tau)} \beta(i,j) \cdot \ba_{\tau}^{(i,j)} \in \mathbb C^{n_\tau^2},
\end{equation}
with $n_\tau = \prod_{i \in L(\tau)} n_i$.
Note that the definition~\eqref{eq:defaij} of $\ba_{\tau}^{(i,j)}$ ensures that the single-site operators $\A_i$ and $\A_j$ continue to act on the original sites $i$ and $j$, respectively.
For $\bar \tau = \tau$, all leaves are included and $\bh_{\bar \tau} = \bh$.
For a leaf $\tau = \ell \in \{1, \dots, d\}$, we set $\bh_\ell = \textbf{0}$.
 
The following lemma is a simple but crucial observation that will allow us to construct TTNO representations recursively.
To simplify the notation, we introduce the compact notation  
$\be_{\tau} := \bigotimes_{\ell \in L(\tau)} \be_{\ell}$.
for an arbitrary tree $\tau$.
\begin{lemma} \label{lemma:part} For $\tau = (\tau_1,\tau_2) \in \mathcal T(\bar\tau)$, the vector $\bh_\tau$ defined in~\eqref{eq:htau} satisfies
\begin{equation} \label{eq:binaryrecursion}
 \bh_\tau = \be_{{\tau_2}} \otimes \bh_{\tau_1} + \bh_{\tau_2} \otimes \be_{{\tau_1}}
 +\sum_{i \in L(\tau_1)\atop j \in L(\tau_2)} \beta(i,j) \cdot \ba_{\tau_2}^{(j)} \otimes \ba_{\tau_1}^{(i)},
\end{equation}
with $\ba_{\tau_1}^{(i)}, \ba_{\tau_2}^{(j)}$ defined as in~\eqref{eq:defai}.
\end{lemma}
\begin{proof} 
The disjoint union $L(\tau) = L(\tau_1) \mathbin{\dot{\cup}} L(\tau_2)$ induces the partition 
\begin{equation} \label{eq:subsets}
 L(\tau)\times L(\tau) = L(\tau_1)\times L(\tau_1) \cup L(\tau_2)\times L(\tau_2) \cup L(\tau_1)\times L(\tau_2) \cup L(\tau_2)\times L(\tau_1).
\end{equation}
We now consider the corresponding division of the sum~\eqref{eq:htau} defining $\bh_\tau$. The first subset only considers terms in the sum for which $(i,j)\in L(\tau_1) \times L(\tau_1)$. Using that~\eqref{eq:defai}
implies $\ba_{\tau}^{(i,j)} = \be_{\tau_2} \otimes \ba_{\tau_1}^{(i,j)}$ for $i,j \in L(\tau_1)$, we obtain for this part of the Hamiltonian that
\[
\sum_{i < j \atop i,j\in L(\tau_1)} \beta(i,j) \cdot \ba_{\tau}^{(i,j)} = 
\be_{\tau_2} \otimes \sum_{i < j \atop i,j\in L(\tau_1)} \beta(i,j) \cdot \ba_{\tau_1}^{(i,j)} = \be_{{\tau_2}} \otimes \bh_{\tau_1}\, , 
\]
matching the first term in~\eqref{eq:binaryrecursion}.
Similarly, the second subset $L(\tau_2)\times L(\tau_2)$ in~\eqref{eq:subsets} yields the second term $\bh_{\tau_2} \otimes \be_{{\tau_1}}$ in~\eqref{eq:binaryrecursion}. The third subset in~\eqref{eq:subsets} directly corresponds to the third term in~\eqref{eq:binaryrecursion}, while the fourth subset does not contribute any terms because no index pair $(i,j)\in L(\tau_2) \times L(\tau_1)$ satisfies the condition $i < j$.
\end{proof}

For theoretical purposes, the following generalization of Lemma~\ref{lemma:part} is of interest, which splits the tree into a subtree (or leaf) $\tau$ and its complement $\bar \tau \setminus \tau$, which is the tree that is obtained by removing the subtree $\tau$ from $\bar \tau$. Compared to the setting of Lemma~\ref{lemma:part}, this situation is more complicated because the leaves of $\tau$ are not necessarily smaller than the leaves of $\bar \tau$. To conveniently account for this, we introduce the symmetrized interaction matrix
\begin{equation} \label{eq:symmbeta}
 \boldsymbol{\beta}_s = \boldsymbol{\beta}+\boldsymbol{\beta}^\top,
\end{equation}
that is, $\boldsymbol{\beta}_s(i,j) = \boldsymbol{\beta}(i,j)$ for $i < j$, $\boldsymbol{\beta}_s(i,j) = \boldsymbol{\beta}(j,i)$ for $i > j$, and $\boldsymbol{\beta}_s(i,i) = 0$.

\begin{lemma} \label{lemma:split}
 For $\tau \in \mathcal T(\bar\tau) \cup L(\bar\tau)$, consider a permutation of the modes that puts the leaves of $\tau$ first:
 \[
  \underbrace{\ell_1,\ell_1 + 1,\ldots,\ell_2-1}_{L(\tau)},\underbrace{1,2,\ldots,\ell_1-1,\ell_2,\ell_2 + 1,\ldots,d}_{L(\bar \tau \setminus \tau)}.
 \]
 Let $\widehat H_{\tau,\bar \tau\setminus \tau}$ denote the corresponding permutation of the matrix representation $\widehat H$ from~\eqref{eq:hamiltonian}. Then the vectorization of $\widehat H_{\tau,\bar \tau\setminus \tau}$
 takes the form
\begin{equation} \label{eq:splitting}
 \bh_{\tau,\bar \tau\setminus \tau}  = \be_{{\bar \tau \setminus \tau}} \otimes \bh_{\tau} + \bh_{\bar \tau \setminus \tau} \otimes \be_{{\tau}}
 +\sum_{i \in L(\tau)\atop j \in L(\bar \tau \setminus \tau)} \beta_s(i,j) \cdot \ba_{\bar \tau \setminus \tau}^{(j)} \otimes  \ba_{\tau}^{(i)}  \in \C^{  n_\tau^2 \cdot n_{\bar \tau \setminus \tau}^2},
\end{equation}
where $\bh_{\bar \tau \setminus \tau} \in  \C^{n_{\bar \tau \setminus \tau}^2}$ is the vectorization of the Hamiltonian that only considers interactions between sites contained in $L(\bar \tau \setminus \tau)$ 
and $n_{\bar \tau \setminus \tau} = \prod_{j \in L(\bar \tau \setminus \tau)} n_j^2$. The vectors
$\ba_{\tau}^{(i)}$, $\ba_{\bar \tau \setminus \tau}^{(j)}$ are defined as in~\eqref{eq:defai}.
\end{lemma}
\begin{proof}
The proof follows along the lines of Lemma~\ref{lemma:part} using, in analogy to~\eqref{eq:subsets}, the partition induced by the disjoint union 
$L(\tau) = L(\tau) \mathbin{\dot{\cup}} L(\bar \tau \setminus \tau)$:
\[
 \{1,\ldots,d\} \times \{1,\ldots,d\} = L(\tau)\times L(\tau) \cup L(\bar \tau\setminus \tau)\times L(\bar \tau\setminus \tau) \cup L(\tau)\times L(\bar \tau\setminus \tau) \cup L(\bar \tau\setminus \tau)\times L(\tau).
\]
As explained in the proof of Lemma~\ref{lemma:part}, the first two terms of this partition match the first two terms of~\eqref{eq:splitting}. Using that $\beta(i,j) = 0$ for $i\ge j$, the last two terms give rise to
\begin{align*}
 &\sum_{i \in L(\tau), j \in L(\bar \tau \setminus \tau), i < j} \beta(i,j) \cdot \ba_{\bar \tau \setminus \tau}^{(j)} \otimes  \ba_{\tau}^{(i)} 
 + \sum_{i \in L(\bar \tau \setminus \tau), j \in L(\tau), i < j} \beta(i,j) \cdot \ba_{\bar \tau \setminus \tau}^{(i)} \otimes  \ba_{\tau}^{(j)}  \\
 =& \sum_{i \in L(\tau), j \in L(\bar \tau \setminus \tau), i < j} \beta(i,j) \cdot \ba_{\bar \tau \setminus \tau}^{(j)} \otimes  \ba_{\tau}^{(i)} 
 + \sum_{i \in L(\tau), j \in L(\bar \tau \setminus \tau), i > j} \beta(j,i) \cdot \ba_{\bar \tau \setminus \tau}^{(j)} \otimes  \ba_{\tau}^{(i)}  \\
 = & \sum_{i \in L(\tau), j \in L(\bar \tau \setminus \tau)} \beta_s(i,j) \cdot \ba_{\bar \tau \setminus \tau}^{(j)} \otimes  \ba_{\tau}^{(i)}.
\end{align*}

\end{proof}

A consequence of Lemma~\ref{lemma:split}, Theorem~\ref{thm:TTNO_unstruct} below is one of the main results of this work. It establishes a TTNO (representation) rank that grows linearly with $d$, instead of the quadratic growth attained by the naive construction mentioned in the beginning of this section. For a balanced binary tree $\bar \tau$, the TTNO rank is bounded by $\lfloor d/2 \rfloor+2$.
\begin{theorem} \label{thm:TTNO_unstruct}
    Let $\widehat H$ be the linear operator defined by \eqref{eq:hamiltonian} and let $\bar \tau$ be a binary dimension tree.
    Then $\widehat H$  admits a TTNO representation $H$ such that the ranks $r_\tau$ satisfy $r_\ell = 2$ at every leaf $\ell \in L(\bar \tau)$ and 
    \begin{equation} \label{eq:rtausubmatrix}
     r_\tau = 2 + \rank( \boldsymbol{\beta}_s(\tau,\bar \tau \setminus \tau) ) \le 2 + d_\tau, \quad \forall \tau \in T(\bar \tau) \setminus \bar \tau,
    \end{equation}
    where $d_\tau$ denotes the cardinality of $L(\tau)$ and $\boldsymbol{\beta}_s(\tau,\bar \tau \setminus \tau)$ is the $d_\tau \times (d-d_\tau)$ submatrix obtained by selecting the rows in $L(\tau)$ and the columns in $\{1,\ldots,d\} \setminus L(\tau)$ of the symmetrized interaction matrix $\boldsymbol{\beta}_s$ from~\eqref{eq:symmbeta}.
\end{theorem} 
\begin{proof} By Definition~\ref{def:TTNO} and Theorem~\ref{thm:ttnbounds}, $H \in \C^{n_1^2\times \cdots \times n_d^2}$ admits a TTNO with the rank $r_\tau$ given by the rank of the matricization $\mat_{L(\tau)}(H)$, which maps the modes in $L(\tau)$ to row indices and the other modes to column indices. By Lemma~\ref{lemma:split}, we have that
\begin{equation} \label{eq:matltau}
 \mat_{L(\tau)}(H) = \bh_{\tau} \be_{{\bar \tau \setminus \tau}}^\top + \be_{{\tau}} \bh^\top_{\bar \tau \setminus \tau} 
 +\sum_{i \in L(\tau)\atop j \in L(\bar \tau \setminus \tau)} \beta_s(i,j) \cdot \ba_{\tau}^{(i)} \big( \ba_{\bar \tau \setminus \tau}^{(j)} \big)^\top.
\end{equation}
For a leaf $\tau = \ell$, this simplifies to 
$\ba_{\ell} \be_{{\bar \tau \setminus \ell}}^\top + \be_{\ell} \bh^\top_{\bar \tau \setminus \tau} 
 +\sum_{j \not= \ell} \beta_s(\ell,j) \cdot \ba_{\ell} \big( \ba_{\bar \tau \setminus \ell}^{(j)} \big)^\top$, which clearly has rank at most $2$, thus establishing $r_\ell = 2$.
 For a subtree $\tau \in T(\bar\tau)$ we can rewrite the third term in~\eqref{eq:matltau} as a matrix product of three matrices, the matrix containing the columns $\ba_{\tau}^{(i)}$, the 
 matrix $\boldsymbol{\beta}_s(\tau,\bar \tau \setminus \tau)$, and the matrix containing the rows $\big( \ba_{\bar \tau \setminus \tau}^{(j)} \big)^\top$. As the rank of this matrix product is bounded by 
 the rank of $\boldsymbol{\beta}_s(\tau,\bar \tau \setminus \tau)$, this establishes~\eqref{eq:rtausubmatrix}.
\end{proof}

We conclude this section by providing the explicit construction of basis matrices and transfer tensor for a TTNO that achieves the upper bound $r_\tau = 2 + d_\tau$ of Theorem~\ref{thm:TTNO_unstruct} for an unstructured interaction matrix $\boldsymbol{\beta}$. For this purpose, let us first consider $\tau = (\tau_1, \tau_2) \in \mathcal{T}(\bar{\tau}) \setminus \{ \bar \tau\}$, where neither $\tau_1$ nor $\tau_2$ is a leaf. Enumerating the leaves $L(\tau_1) = \{\ell_1, \ldots, \ell_2-1\}$ and $L(\tau_2) = \{\ell_2, \ldots, \ell_3-1\}$,  we introduce
%
\begin{align*}
&\U_{\tau_1} := \begin{bmatrix}
                   \be_{\tau_1} & \bh_{\tau_1} & \ba_{\tau_1}^{(\ell_1)} & \cdots & \ba_{\tau_1}^{(\ell_2-1)}
                  \end{bmatrix} \in \C^{n_{\tau_1}^2 \times (2 +d_{\tau_1})}, \\
&\U_{\tau_2} := \begin{bmatrix}
                   \be_{\tau_2} & \bh_{\tau_2} & \ba_{\tau_2}^{(\ell_2)} & \cdots & \ba_{\tau_2}^{(\ell_3-1)}
                  \end{bmatrix} \in \C^{n_{\tau_2}^2 \times (2 +d_{\tau_2})}.
\end{align*}
It follows from the result~\eqref{eq:binaryrecursion} of Lemma~\ref{lemma:part} that each column in the corresponding matrix $\U_{\tau}$ for $\tau = (\tau_1,\tau_2)$ can be obtained as a linear combination of Kronecker products between columns of $\U_{\tau_2}$ and columns of $\U_{\tau_1}$.
In other words, there exists a transfer tensor 
$C_{\tau} \in \C^{(2+d_{\tau_1}) \times (2+d_{\tau_2} )  \times (2+d_{\tau})}$ such that  
\[
 \U_{\tau} = \begin{bmatrix}
                   \be_{\tau} & \bh_{\tau} & \ba_\tau^{(\ell_1)} & \cdots & \ba_\tau^{(\ell_3-1)}
                  \end{bmatrix} = (\U_{\tau_2}\otimes \U_{\tau_1}) \mat_{1,2}(C_{\tau}) \in \C^{n_{\tau}^2 \times (2 +d_\tau)} ,
\]
where we recall that $n_\tau = n_{\tau_1} n_{\tau_2}$. To determine the entries of $C_{\tau}$, it is helpful to reshape each column of $\U_{\tau}$ as an $n_{\tau_1} \times n_{\tau_2}$ matrix. Using $\U^{(k)}_{\tau}$ to denote the matrix corresponding to the $k$th column, it follows from basic properties of the Kronecker product that
\[
 \U^{(k)}_{\tau} = \U_{\tau_1} C_{\tau}(:,:,k) \U_{\tau_2}^\top,
\]
where $C_{\tau}(:,:,k) \in \C^{(2+d_{\tau_1}) \times (2+d_{\tau_2} ) }$ is the $k$th frontal slice of $C_{\tau}$. For example, the first column $\be_{\tau}$ corresponds to $\U^{(1)}_{\tau} = \be_{\tau_1} \be_{\tau_2}^\top$ and thus the first slice of
$C_{\tau}$ is zero except for the entry $1$ at position $(1,1)$. Continuing in this manner, we obtain that the first two frontal slices of $C_{\tau}$ are 
\begin{align}
\left[\begin{array}{@{}c|c@{}}
  \begin{matrix}
  1 & 0 \\
  0 & 0
  \end{matrix}
  & \textbf{0} \\
\hline 
  \textbf{0} &
  \begin{matrix}
  \textbf{0} 
  \end{matrix}
\end{array}\right], \quad
\left[\begin{array}{@{}c|c@{}}
  \begin{matrix}
  0 & 1 \\
  1 & 0
  \end{matrix}
  & \textbf{0} \\
\hline 
  \textbf{0} &
  \begin{matrix}
  \boldbeta(\tau_1, \tau_2) 
  \end{matrix}
\end{array}\right], \quad
 \, \label{eq:first_slices}
\end{align}
where $\boldsymbol{\beta}(\tau_1,\tau_2)$ is the submatrix of $\boldsymbol{\beta}$ obtained by selecting the rows in $L(\tau_1)$ and the columns in $L(\tau_2)$. 
The remaining $d_\tau = d_{\tau_1} + d_{\tau_2}$ frontal slices of $C_{\tau}$ are given by
\begin{align}
\left[\begin{array}{@{}c|c@{}}
  \begin{matrix}
  0 & 0 \\
  0 & 0
  \end{matrix}
  & \textbf{0} \\
\hline 
  \textbf{0} &
  \begin{matrix}
  \bu_3 \bu_1^\top 
  \end{matrix}
\end{array}\right],
\dots,
\left[\begin{array}{@{}c|c@{}}
  \begin{matrix}
  0 & 0 \\
  0 & 0
  \end{matrix}
  & \textbf{0} \\
\hline 
  \textbf{0} &
  \begin{matrix}
  \bu_{2+d_{\tau_1}} \bu_1^\top 
  \end{matrix}
\end{array}\right], 
\left[\begin{array}{@{}c|c@{}}
  \begin{matrix}
  0 & 0 \\
  0 & 0
  \end{matrix}
  & \textbf{0} \\
\hline 
  \textbf{0} &
  \begin{matrix}
  \bu_1 \bu_3^\top 
  \end{matrix}
\end{array}\right], \dots ,
\left[\begin{array}{@{}c|c@{}}
  \begin{matrix}
  0 & 0 \\
  0 & 0
  \end{matrix}
  & \textbf{0} \\
\hline 
  \textbf{0} &
  \begin{matrix}
  \bu_1 \bu_{2+d_{\tau_2}}^\top 
  \end{matrix}
\end{array}\right], \label{eq:last_slices}
\end{align}
where $\bu_i$ is the vector (of appropriate length) with $1$ at entry $i$ and zeros everywhere else.

When $\tau_1$ and/or $\tau_2$ are leaves, then one or both of the first two terms in~\eqref{eq:binaryrecursion} vanish, i.e., $\bh_\ell = \mathbf{0}$. Consequently, the basis matrices do not need to account for these terms and take the form
\begin{equation} \label{eq:basismatrices}
 \U_\ell := \begin{bmatrix} \be_\ell & \vect(\A_\ell) \end{bmatrix} \in \C^{n_\ell^2\times 2}.
\end{equation}
Transfer tensors involving leaves need to be adjusted accordingly: It suffices to take into account the first slice of \eqref{eq:first_slices} and a modified version of the second where one or both of the ones in the upper-right block are set to zero. At the root tree $\bar \tau$, it suffices to form $\bh_{\bar \tau}$, which is equal to $\bh$ by construction. Hence, the transfer tensor $C_{\bar \tau}$ becomes a matrix equal to the second matrix in~\eqref{eq:first_slices}, while setting the upper left block in the second slice to zero.

Our construction satisfies $r_\ell = 2$ for each leaf $\ell$ because of~\eqref{eq:basismatrices} and $r_\tau = d_\tau + 2$ for every subtree $\tau \not= \bar \tau$ because of the size of the transfer tensor $C_\tau$.

\subsection{General trees}
\label{sec:generaltrees}

The results of Section~\ref{subsec:unstruct_case} extend to general trees, which allow for an arbitrary number of children at each node. In the following, we briefly sketch the construction of the TTNO in this general case.  While the basis matrices are defined in the same way as in~\eqref{eq:basismatrices}, the construction of the transfer tensors becomes slightly more technical. Considering a subtree $\tau = (\tau_1,\ldots,\tau_m)$,
we partition 
\[
 \boldsymbol{\beta}(\tau,\tau) = \begin{bmatrix}
                  \boldsymbol{\beta}(\tau_1,\tau_1) & \cdots &  \boldsymbol{\beta}(\tau_1,\tau_m) \\
                 & \ddots & \vdots \\
                 & &  \boldsymbol{\beta}(\tau_m,\tau_m) 
                \end{bmatrix}.
\]
The corresponding partition of the sum~\eqref{eq:htau} is now given by  \begin{equation}\label{eq:partgeneraltree}
 \bh_\tau = \sum_{i = 1}^m \be_{\tau_m} \otimes \cdots \otimes \be_{\tau_{i+1}} \otimes \bh_{\tau_i} \otimes \be_{\tau_{i-1}}\otimes \cdots \otimes \be_{\tau_1} + \sum_{i < j}^m \sum_{k\in L(\tau_i) \atop \ell \in L(\tau_j)} \beta(k, \ell) \cdot \ba_{\tau_j}^{(\ell)} \otimes \ba_{\tau_i}^{(k)}.
\end{equation}
Enumerating $L(\tau_i) = \{\ell_i,\ell_i+1,\ldots, \ell_{i+1}-1\}$, we proceed as before and set 
\[
 \U_{\tau_i} = \begin{bmatrix}
                   \be_{\tau_i} & \bh_{\tau_i} & \ba^{(\ell_i)} & \cdots & \ba^{(\ell_{i+1}-1)}
                  \end{bmatrix}, \qquad  i = 1,\ldots, m.
\] In analogy to the case of a binary tree, the decomposition~\eqref{eq:partgeneraltree} implies that there exists a transfer tensor $C_{\tau}$ of order $m+1$ such that
\[
 \U_{\tau} = \begin{bmatrix}
                   \be_{\tau} & \bh_{\tau} & \ba^{(\ell_1)} & \cdots & \ba^{(\ell_{m+1}-1)}
                  \end{bmatrix} = ( \U_{\tau_m} \otimes \cdots \otimes \U_{\tau_1} ) \mat_{1:m}(C_{\tau}^A).
\]
In summary, the discretized Hamiltonian~\eqref{eq:hamiltonian} also admits a TTNO representation for a general tree $\bar \tau$, with the ranks $r_\ell = 2$ for every leaf $\ell$ and $r_\tau =2+d_\tau$ for every subtree $\tau \neq \bar \tau$. Allowing for more children can be used to decrease $d_\tau$ and thus the ranks, at the expense of increasing the order of the transfer tensors.

\section{Compressed TTNO via HSS decomposition} \label{sec:ttnohss}

In this section, we will briefly introduce HSS matrices and establish their connection to TTNOs of Hamiltonians describing pairwise interactions.

\subsection{Hierarchical Semi-Separable (HSS) matrices}

We will use the concept of Hierarchically Semi-Separable (HSS) matrices, as defined in \cite{MR2759603}, to represent the $d\times d$ interaction matrix $\boldsymbol{\beta}$. For this purpose, a binary dimension tree $\bar \tau$ (see Definition \ref{def:dimtree}) is used to to recursively block-partition a $d\times d$ matrix $\boldsymbol{\beta}$. At the root $\bar \tau = (\tau_1,\tau_2)$, this corresponds to partitioning
\[
 \boldsymbol{\beta} = \boldsymbol{\beta}(\bar \tau,\bar \tau) = \begin{bmatrix}
 \boldsymbol{\beta}(\tau_1,\tau_1) & \boldsymbol{\beta}(\tau_1,\tau_2) \\
 & \boldsymbol{\beta}(\tau_2,\tau_2)
         \end{bmatrix}.
\]
This partitioning is recursively repeated for $\boldsymbol{\beta}(\tau_1,\tau_1)$ and $\boldsymbol{\beta}(\tau_2,\tau_2)$ using the subtrees $\tau_1$ and $\tau_2$, respectively.
Figure \ref{fig:hss_dec_TTN_TT} provides a graphical representation for different dimension trees. Note that HSS matrices are commonly used with (nearly) balanced binary trees; the degenerate tree shown in the right of Figure~\ref{fig:hss_dec_TTN_TT} is closely related to the notion of quasi-separable matrices; see also~\cite{KMR19} for a discussion.

For $\boldsymbol{\beta}$ to be an HSS matrix it is assumed that the off-diagonal block $\boldsymbol{\beta}(\tau_1,\tau_2)$ for every subtree $\tau = (\tau_1,\tau_2)$ admits a (low-rank) factorization
\begin{equation} \label{eq:betapart}
    \boldsymbol{\beta}(\tau_1,\tau_2) = \V_{\tau_1} \bS_{\tau_1,\tau_2} \V_{\tau_2}^*, \ \ \ \ \ \ \V_{\tau_1} \in \C^{d_{\tau_1} \times k_{\tau_1}}, \ \bS_{\tau_1,\tau_2} \in \C^{k_{\tau_1} \times k_{\tau_2}}, \ \V_{\tau_2} \in \C^{d_{\tau_2} \times k_{\tau_1}},
\end{equation}
for some (small) integers $k_{\tau_1}$, $k_{\tau_2}$. Recall that $d_{\tau_1}$, $d_{\tau_2}$ denote the cardinality of $L(\tau_1)$, $L(\tau_2)$.
Additionally, the HSS decomposition requires that the basis matrices $\V_{\tau_i}$ are nested across all levels. For each tree $\tau=(\tau_1, \tau_2)$, this structure is characterized by the existence of a translation operator $\bR_{\tau} \in \mathbb{C}^{(k_{\tau_1}+k_{\tau_2}) \times k_\tau}$ such that
\begin{align} \label{eq:nested}
    \V_{\tau} = \begin{pmatrix}
        \V_{\tau_1} & 0 \\
        0 & \V_{\tau_2}
    \end{pmatrix} \bR_{\tau} \in \C^{d_\tau \times k_\tau} .
\end{align}
The HSS rank is the largest value of $k_\tau$ across all levels. To represent an HSS matrix $\boldsymbol{\beta}$ it suffices to store the middle factor $\bS_{\tau_1,\tau_2}$ from~\eqref{eq:betapart} for all siblings $(\tau_1,\tau_2)$ and the translation matrices $\bR_{\tau}$, assuming that basis matrices on the leaf level are normalized to $V_\ell = 1$.
\begin{figure} \label{fig:hss_dec_TTN_TT}
\begin{center}
	\hspace{0.5cm}
	\begin{tikzpicture}
		\coordinate (S1) at (0,0);
		\coordinate (S2) at (0,4);
		\coordinate (S3) at (4,4);
		\coordinate (S4) at (4,0);
		
		\draw (S1) -- (S2) -- (S3) -- (S4) -- (S1);
		
		\coordinate (S11) at (2,0);
		\coordinate (S12) at (4,2);
		\coordinate (S13) at (2,4);
		\coordinate (S14) at (0,2);
		
		\draw (S11) -- (S13);
		\draw (S12) -- (S14);
		
		\coordinate (S21) at (1,2);
		\coordinate (S22) at (2,3);
		\coordinate (S23) at (1,4);
		\coordinate (S24) at (0,3);
		\draw (S21) -- (S23);
		\draw (S22) -- (S24);
		
		\coordinate (S211) at (3,0);
		\coordinate (S212) at (2,1);
		\coordinate (S213) at (3,2);
		\coordinate (S214) at (4,1);
		\draw (S211) -- (S213);
		\draw (S212) -- (S214);
		
		\coordinate (T11) at (0.5,3);
		\coordinate (T12) at (0,3.5);
		\coordinate (T13) at (0.5,4);
		\coordinate (T14) at (1,3.5);
		\draw (T11) -- (T13);
		\draw (T12) -- (T14);
		
		\coordinate (T21) at (1.5,2);
		\coordinate (T22) at (1,2.5);
		\coordinate (T23) at (1.5,3);
		\coordinate (T24) at (2,2.5);
		\draw (T21) -- (T23);
		\draw (T22) -- (T24);
		
		\coordinate (T31) at (2.5,1);
		\coordinate (T32) at (2,1.5);
		\coordinate (T33) at (2.5,2);
		\coordinate (T34) at (3,1.5);
		\draw (T31) -- (T33);
		\draw (T32) -- (T34);
		
		\coordinate (T41) at (3.5,0);
		\coordinate (T42) at (3,0.5);
		\coordinate (T43) at (3.5,1);
		\coordinate (T44) at (4,0.5);
		\draw (T41) -- (T43);
		\draw (T42) -- (T44);
		
	\end{tikzpicture}
	\hspace{1.5cm}
	\begin{tikzpicture}
		\coordinate (S1) at (0,0);
		\coordinate (S2) at (0,4);
		\coordinate (S3) at (4,4);
		\coordinate (S4) at (4,0);
		
		\draw (S1) -- (S2) -- (S3) -- (S4) -- (S1);
		
		\foreach \i in {0.5,1,...,3.5}
		{
		\draw (\i,0) -- (\i,4.5 -\i);
		\draw (-0.5 + \i,4 - \i) -- (4,4 - \i); 
		}
		
		\end{tikzpicture}
	\end{center}
	\vspace{0.3cm}
	\begin{center}
		\begin{tikzpicture}[every node/.style={sibling distance=10mm},level 1/.style={sibling distance=27mm},level 2/.style={sibling distance=13mm},level 3/.style={sibling distance=8mm},level 4/.style={sibling distance=5mm},level distance = 6mm]
		\node[circle,draw,scale=0.8]  { }
		child { node[circle,draw,scale=0.8] {} 
			child{ node[circle,draw,scale=0.8] {} 
				child{ node[circle,draw,scale=0.8] {} }
				child{ node[circle,draw,scale=0.8] {}	
					}}
			child{ node[circle,draw,scale=0.8] {} 
				child{ node[circle,draw,scale=0.8] {}}
				child{ node[circle,draw,scale=0.8] {} }}}
		child { node[circle,draw,scale=0.8] {} 
			child{ node[circle,draw,scale=0.8] {} 
				child{ node[circle,draw,scale=0.8] {} }
				child{ node[circle,draw,scale=0.8] {}	
					}}
			child{ node[circle,draw,scale=0.8] {} 
				child{ node[circle,draw,scale=0.8] {}}
				child{ node[circle,draw,scale=0.8] {} }}};
	\end{tikzpicture}
	\hspace{1.5cm}
	\begin{tikzpicture}[every node/.style={sibling distance=2mm},,level 1/.style={sibling distance=8mm},level 2/.style={sibling distance=8mm},level 3/.style={sibling distance=8mm},level 4/.style={sibling distance=8mm},level 5/.style={sibling distance=8mm},level 6/.style={sibling distance=8mm},level distance = 4mm]
		\node[circle,draw,scale=0.8]  { }
			child{ node[circle,draw,scale=0.8]{} }
			child{ node[circle,draw,scale=0.8]{}{
				child{ node[circle,draw,scale=0.8]{} }
				child{ node[circle,draw,scale=0.8]{} {
					child{ node[circle,draw,scale=0.8]{} }
					child{ node[circle,draw,scale=0.8]{} {
						child{ node[circle,draw,scale=0.8]{} }
						child{ node[circle,draw,scale=0.8]{} {
							child{ node[circle,draw,scale=0.8]{} }
							child{ node[circle,draw,scale=0.8]{} {
								child{ node[circle,draw,scale=0.8]{} }
								child{ node[circle,draw,scale=0.8]{}{
									child{ node[circle,draw,scale=0.8]{} }
								child{ node[circle,draw,scale=0.8]{} }
			}}}}}}}}}}}};
	\end{tikzpicture}
	\end{center}
 \caption{\label{fig:lemmapart} Different recursive block-partitions of an $8 \times 8$ interaction matrix $\boldbeta$. Left: Recursive block-partition corresponding to a balanced binary tree. Right: Recursive block-partition corresponding to a degenerate tree.}
\end{figure}
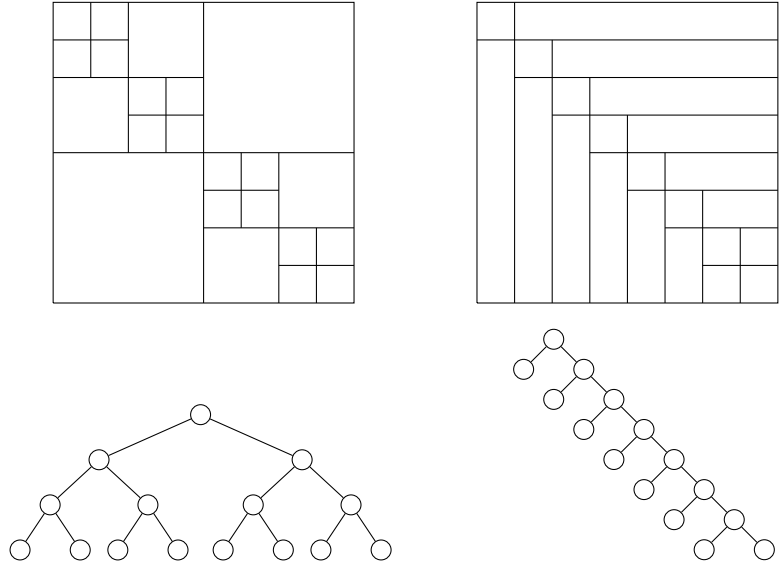
\begin{remark} \label{remark:hss} \rm 
It is important to note that~\eqref{eq:betapart} enforces identical left and right factors, which is not required in the usual definition of HSS matrices. However, it is not difficult to see that the above definition of an HSS matrix $\boldsymbol{\beta}$ coincides with the usual definition applied to the \emph{symmetrized} interaction matrix $\boldsymbol{\beta}_s = \boldsymbol{\beta} + \boldsymbol{\beta}^\top$. In particular, the results from~\cite{MR2759603} imply that the smallest $k_\tau$ for which $\boldsymbol{\beta}$ admits an HSS decomposition in the sense of~\eqref{eq:betapart}--\eqref{eq:nested} is given by
\[
 k_\tau = \rank\big( \boldsymbol{\beta}_s(\tau,\bar \tau \setminus \tau) \big), \quad \forall \tau \in T(\bar \tau).
\]
The submatrix $\boldsymbol{\beta}_s(\tau,\bar \tau \setminus \tau)$ is commonly referred to as an HSS block row of $\boldsymbol{\beta}_s$.
\end{remark}

\subsection{TTNO from HSS} \label{eq:ttnofromhss}

In the following, we derive a TTNO for the Hamiltonian from an HSS decomposition of the interaction matrix $\boldbeta$.

\begin{corollary} \label{cor:HSS_construction}
     Let $\widehat H$ be the linear operator defined by \eqref{eq:hamiltonian}, and let $\bar \tau$ be a binary dimension tree. If $\boldsymbol{\beta}$ admits an HSS decomposition~\eqref{eq:betapart}--\eqref{eq:nested} then there exists a TTNO representation $H$ of $\widehat H$ with representation ranks $r_\tau = 2+k_\tau$ for every subtree $\tau \in T(\bar \tau)$ and 
     $r_\ell = 2$ for every leaf $\ell \in L(\bar \tau)$.
\end{corollary}
\begin{proof}
 This result is a direct consequence of Theorem~\ref{thm:TTNO_unstruct} combined with Remark~\ref{remark:hss}. 
\end{proof}

In the rest of this section, we provide an explicit construction that realizes the representation ranks established by Corollary~\ref{cor:HSS_construction} from an HSS decomposition~\eqref{eq:betapart}--\eqref{eq:nested} of $\boldsymbol{\beta}$. For this purpose, let $\tau = (\tau_1, \tau_2) \in \mathcal{T}(\bar{\tau}) \setminus \{ \bar \tau\}$ such that neither $\tau_1$ nor $\tau_2$ is a leaf. Recalling the result of Lemma \ref{lemma:part}, the vectorization of the part of the Hamiltonian $\widehat H$ corresponding to $\tau$ takes the form\begin{align*}
        \bh_\tau = \be_{{\tau_2}} \otimes \bh_{\tau_1} + \bh_{\tau_2} \otimes \be_{{\tau_1}}
 +\sum_{i \in L(\tau_1)\atop j \in L(\tau_2)} \beta(i,j) \cdot  \ba_{\tau_2}^{(j)} \otimes \ba_{\tau_1}^{(i)}.
    \end{align*}
    Introducing the short-hand notation 
    \begin{align*}
    \ba_{\tau_1} := [\ba_{\tau_1}^{(\ell_1)} \ \cdots \   \ba_{\tau_1}^{(\ell_2-1)}] \in \C^{n_{\tau_1}^2 \times d_{\tau_1}}, \quad 
        \ba_{\tau_2} := [\ba^{(\ell_2)}_{\tau_2} \ \cdots \  \ba_{\tau_2}^{(\ell_3-1)}] \in \C^{n_{\tau_2}^2 \times d_{\tau_2}}
    \end{align*}
    allows us to rewrite the third summand as 
    \begin{align}
         \sum_{i \in L(\tau_1)\atop j \in L(\tau_2)} \beta(i,j) \cdot  \ba_{\tau_2}^{(j)} \otimes \ba_{\tau_1}^{(i)} =\left( \ba_{\tau_2} \otimes \ba_{\tau_1} \right) \vect(\boldbeta(\tau_1, \tau_2) ).
         \label{eq:proof_hss_construction}
    \end{align}
    The HSS structure of the interaction matrix $\boldbeta$ induces a low-rank structure on the off-diagonal block $\boldbeta(\tau_1, \tau_2) = \textbf{V}_{\tau_1} \textbf{S}_{\tau_1 ,\tau_2} \textbf{V}_{\tau_2}^*$; see~\eqref{eq:betapart}.
    Hence, we obtain that~\eqref{eq:proof_hss_construction} can be rewritten as  
    \begin{equation} \label{eq:compressedSum}
        \left(\ba_{\tau_2} \otimes \ba_{\tau_1}\right) \vect(\beta(\tau_1, \tau_2) ) 
        = \left( \ba_{\tau_2} \textbf{V}_{\tau_2} \otimes \ba_{\tau_1} \textbf{V}_{\tau_1} \right) \vect(\textbf{S}_{\tau_1 ,\tau_2}).
    \end{equation}
    In turn, it suffices to consider the compressed bases 
    \[
        \tilde \ba_{\tau_1} := \ba_{\tau_1} \textbf{V}_{\tau_1} \in \C^{n_{\tau_1}^2 \times k_{\tau_1}},\quad \tilde \ba_{\tau_2} := \ba_{\tau_2} \textbf{V}_{\tau_2} \in \C^{n_{\tau_2}^2 \times k_{\tau_2}}.\]
     We proceed as in the construction of Section~\ref{subsec:unstruct_case} for the unstructured case, but using compressed bases. For this purpose, we introduce
    \begin{align*}
        &\widetilde \U_{\tau_1} := [\be_{\tau_1} \  \bh_{\tau_1} \  \tilde \ba_{\tau_1} ] \in \C^{n_{\tau_1}^2 \times (2 +k_{\tau_1})}, \quad
        \widetilde \U_{\tau_2} := [\be_{\tau_2} \ \bh_{\tau_2} \   \tilde \ba_{\tau_2} ] \in \C^{n_{\tau_2}^2 \times (2 +k_{\tau_2})}.
    \end{align*}
    We now aim at determining the tensor $\widetilde C_\tau \in \C^{(2+k_{\tau_1}) \times (2+k_{\tau_2}) \times (2+k_{\tau})}$ that transfers these bases to the corresponding basis at the parent node:
    \[
     \widetilde \U_{\tau} \ := [\be_{\tau}   \ \  \bh_{\tau} \ \  \tilde \ba_\tau  ] \in \C^{n_{\tau}^2 \times (2 +k_\tau)},\quad \tilde \ba_{\tau} := \ba_{\tau} \textbf{V}_{\tau}.
    \]

    Similarly to (\ref{eq:first_slices}), the first two frontal slices of $\widetilde C_\tau$ are determined using~\eqref{eq:compressedSum}:
    $$
\left[\begin{array}{@{}c|c@{}}
  \begin{matrix}
  1 & 0 \\
  0 & 0
  \end{matrix}
  & \textbf{0} \\
\hline 
  \textbf{0} &
  \begin{matrix}
  \textbf{0} 
  \end{matrix}
\end{array}\right], \quad
\left[\begin{array}{@{}c|c@{}}
  \begin{matrix}
  0 & 1 \\
  1 & 0
  \end{matrix}
  & \textbf{0} \\
\hline 
  \textbf{0} &
  \begin{matrix}
  \textbf{S}_{\tau_1, \tau_2} 
  \end{matrix}
\end{array}\right].$$
To determine the remaining frontal slices of $\widetilde C_\tau$, we extend~\eqref{eq:last_slices} by first defining the matrix
\[ \widetilde{\textbf{M}} = \left[ \vect( \bu_3 \bu_1^\top),\ \ldots, \vect(\bu_{2+k_{\tau_1}} \bu_1^\top), \vect(\bu_1 \bu_3^\top), \ldots, \vect(\bu_1 \bu^\top_{2+k_{\tau_2}}) \right] \in \C^{k_{\tau_1} k_{\tau_2} \times (k_{\tau_1} + k_{\tau_2}) },\]
where $\bu_i$ again denotes the $i$th unit vector (of appropriate length), with $1$ at entry $i$ and zeros everywhere else. The definition of this matrix ensures that
\begin{equation} \label{eq:compressedTransformation}
\left[ \be_{\tau_2} \otimes \tilde \ba_{\tau_1} \ | \ \tilde \ba_{\tau_2} \otimes \be_{\tau_1} \right]
        =  \left( \widetilde \U_{\tau_2} \otimes \widetilde \U_{\tau_1}   \right) \widetilde{ \textbf{M}} \in \C^{n_\tau^2 \times (k_{\tau_1} + k_{\tau_2})}.
\end{equation}
Together with the nestedness~\eqref{eq:nested} of the HSS basis matrices, this gives
\begin{align*}
    \tilde \ba_\tau 
        = \ba_{\tau}  \textbf{V}_\tau 
        &=   \left[ \be_{\tau_2} \otimes \ba_{\tau_1} \ | \ \ba_{\tau_2} \otimes \be_{\tau_1} \right]
            \left[ \begin{array}{cc} \textbf{V}_{\tau_1} & 0 \\ 0 & \textbf{V}_{\tau_2} \end{array} \right] \textbf{R}_\tau \\
        &= \left[ \be_{\tau_2} \otimes \tilde \ba_{\tau_1} \ | \ \tilde \ba_{\tau_2} \otimes \be_{\tau_1} \right] \textbf{R}_\tau = \left( \widetilde \U_{\tau_2} \otimes \widetilde \U_{\tau_1}   \right) \widetilde{ \textbf{M}} \textbf{R}_\tau \in \C^{n_\tau^2 \times k_\tau}.
\end{align*}
Hence, instead of~\eqref{eq:last_slices}, the remaining slices of $\tilde C_\tau$ are now set to
\begin{align*}
    \left[\begin{array}{@{}c|c@{}}
  \begin{matrix}
  0 & 0 \\
  0 & 0
  \end{matrix}
  & \textbf{0} \\
    \hline  \\[-2ex]
  \textbf{0} & 
  \begin{matrix} 
  \bigl(\widetilde{\textbf{M}} \textbf{R}_\tau \bigr)_i 
  \end{matrix} 
\end{array}\right] \qquad \ i=1,\dots,k_\tau \, ,
\end{align*}
where $(\widetilde{\textbf{M}} \textbf{R}_\tau)_i$ denotes the matricization of the $i$th column . The transfer tensors at the leaves and at the root tree $\bar \tau$ are adjusted accordingly. This concludes our construction of the TTNO from the HSS decomposition of $\boldsymbol{\beta}$ by choosing the basis matrices $\U_\ell$ defined in~\eqref{eq:basismatrices}
and the transfer tensors defined above. The ranks $r_\tau$ of this construction match those of Corollary~\ref{cor:HSS_construction}.

\begin{remark} \rm 
The construction above extends to trees with arbitrarily many children at each node by employing a suitable generalization of the HSS decomposition, which permits the use of more than two subtrees on each level. However, we are not aware of algorithms or even software that cover such a general setting.
\end{remark}

\subsection{Approximation of $\boldsymbol{\beta}$ by an HSS  matrix}

In simple cases, $\boldsymbol{\beta}$ is known to admit an exact HSS decomposition with small ranks. For example, in the case of nearest-neighbor interaction, the only nonzero entries of $\boldsymbol{\beta}$ are $\beta(i,i+1) = 1$ for $i = 1,\ldots,d-1$. In turn, the rank of the HSS block row $\boldsymbol{\beta}_s(\tau,\bar \tau \setminus \tau)$ is bounded by two because this matrix has at most two nonzero entries. By Corollary~\ref{cor:HSS_construction}, this implies that the TTNO rank is bounded by four, which recovers known results from the literature; see, e.g.,~\cite[Example 3.8]{Tobler2012}.

For long-rank interactions, $\boldsymbol{\beta}$ usually does \emph{not} admit an HSS decomposition with small ranks but it can often be well approximated by such a matrix. This approximation is closely related to the low-rank approximation of the HSS block rows $\boldsymbol{\beta}_s(\tau,\bar \tau \setminus \tau)$ from Remark~\ref{remark:hss}. Given a tolerance $\epsilon > 0$, we choose $k$ such that
\begin{equation} \label{eq:epsassumptions}
 \sigma_{k+1}( \boldsymbol{\beta}_s(\tau,\bar \tau \setminus \tau) ) \le \epsilon \| \boldsymbol{\beta}_s(\tau,\bar \tau \setminus \tau)  \|_2,
\end{equation}
where $\sigma_{k+1}(\cdot)$ denotes the $(k+1)$th largest singular value of a matrix. This property is known to hold with $k = \mathcal O( \log(d/\epsilon) )$ for long-range interactions commonly found in the literature, including Coulomb interaction $\beta(i,j) = 1/|i-j|$; see~\cite{Braess2009,LT21}.

When~\eqref{eq:epsassumptions} is satisfied, one can determine an approximation of $\boldsymbol{\beta}$ having HSS rank $k$ and approximation error proportional to $\epsilon$. This follows from applying~\cite[Corollary 4.3]{XXCB14} to the symmetrized interaction matrix $\boldsymbol{\beta}_s$. The result is constructive; the matrix $\boldsymbol{\beta}_{k}$ can be constructed by applying the SVD-based procedures from~\cite{XXCB14,MRK20} to $\boldsymbol{\beta}_s$. 
\begin{lemma} \label{lem:hssErrorBounds}
    Let $\boldbeta \in \C^{d \times d}$ be a strictly upper triangular interaction matrix such that~\eqref{eq:epsassumptions} is satisfied for a binary dimension tree $\bar \tau$ and some $\epsilon  > 0$. Then there exists a strictly upper triangular matrix $\boldsymbol{\beta}_{k} \in \C^{d \times d}$ in HSS decomposition~\eqref{eq:betapart}--\eqref{eq:nested} of HSS rank $k$ such that
    \[
        \| \boldsymbol{\beta} - \boldsymbol{\beta}_k \|_F 
            \le  C h(\bar \tau) \sqrt{k} \| \boldsymbol \beta \|_{\rm F} \cdot  \epsilon\, .
    \]
    is satisfied 
    for some constant $C$, where $h(\bar \tau)$ denotes the height of $\bar \tau$.
\end{lemma}

Using the approximation $\boldsymbol{\beta}_k$ from Lemma~\ref{lem:hssErrorBounds}, a TTNO decomposition of the corresponding Hamiltonian $\widehat H_{k}$ can be cheaply obtained using the procedure described in Section~\ref{eq:ttnofromhss}. The following result shows that $\widehat H_{k}$ is $\mathcal O(\epsilon)$-close to $\widehat H$ in the spectral norm.
\begin{theorem} \label{thm:error_bounds}
Under the setting and assumptions of Lemma~\ref{lem:hssErrorBounds}, let $\widehat H$ be the linear operator defined by \eqref{eq:hamiltonian}. Then
    \[ \widehat H_{k} = \sum_{i < j}^d \beta_k(i,j) \cdot \A^{(i)}\A^{(j)} \in \mathbb C^{(n_1\cdots n_d) \times (n_1\cdots n_d)},\]
    has TTNO rank $2+k$ and satisfies the error bound
    \begin{align*} 
    \norm{ \widehat H -  \widehat H_{k,\epsilon}}_2 
        &\leq  C h(\bar \tau) \sqrt{k} \| \boldsymbol \beta \|_{\rm F} \Big(  \sum_{i < j}^d \| \A_i \|^2_2 \| \A_j \|^2_2 \Big)^{1/2} \cdot \epsilon.
        \end{align*}
\end{theorem}

\begin{proof}
From the triangular inequality and the Cauchy-Schwartz inequality, it follows that
\begin{align*}
  \norm{\widehat H - \widehat H_{k,\epsilon}}_2 
        & = \Big\| \sum_{i < j}^d  ( \beta(i,j)-\beta_{k}(i,j)) \A^{(i)}\A^{(j)} \Big\|_2  
        \leq \sum_{i < j}^d |\beta(i,j)-\beta_{k}(i,j)| \cdot \| \A^{(i)} \|_2 \| \A^{(j)} \|_2 \\
        & = \sum_{i < j}^d |\beta(i,j)-\beta_{k}(i,j)| \cdot \| \A_i \|_2 \| \A_j \|_2 \le \|\boldsymbol{\beta}-\boldsymbol{\beta}_k\|_F \Big(  \sum_{i < j}^d \| \A_i \|^2_2 \| \A_j \|^2_2 \Big)^{1/2}.
\end{align*}
The proof is concluded by applying Lemma~\ref{lem:hssErrorBounds} and Corollary~\ref{cor:HSS_construction}.
\end{proof}

\section{Numerical experiments} \label{sec:numexp}

In this section, we report the results of numerical experiments applying our described construction of TTNOs for diverse quantum spin systems. These experiments were conducted using MATLAB 2018b software along with the \texttt{MATLAB tensor toolboxes}~\cite{TTB_Software, tensorlab3.0} and the \texttt{hm-toolbox}~\cite{MRK20}.

\subsection{Closed quantum spin system}\label{subsec:num_ex1}
In the first example, we consider the operator associated with the Schrödinger equation with long-range unitary dynamics for spin-$\frac{1}{2}$ particles:
\begin{align} \label{eq:closedSpinSystemOperator}
        \mathcal H = \Omega \sum_{k=1}^{d} \sigma_x^{(k)} + \Delta \sum_{k=1}^{d} n^{(k)} + \nu \sum_{i < j} \frac{1}{(j - i)^{\alpha}} n^{(i)} n^{(j)}.
\end{align}
Here, $\sigma_x^{(k)}$ denotes the first $2\times 2$ Pauli matrix acting on the $k$th site, while $n^{(i)}$ and $n^{(j)}$ denote the projectors onto the $i$th and $j$th excited states, respectively. See \cite{sulz2023} for a more detailed description. The parameters $\Omega$, $\Delta$, and $\nu$ are model-related, while the parameter $\alpha$ describes various interaction regimes among the particles:
\begin{itemize}
    \item $\alpha = 0$ encodes an all-to-all interaction;
    \item $\alpha = \infty$ encodes nearest-neighbor interactions;
    \item $0<\alpha<\infty$ encodes long-range interactions. 
\end{itemize}
In quantum physics related settings, values of interest are given by $\alpha = 1$ (Coulomb interaction), $\alpha = 3$ (dipole-dipole interaction) or $\alpha= 6$ (van der Waals interaction)~\cite{saffman2010}. 

In Figure \ref{fig:TTNO_unitary_alpha1}, we observe that the relative error is independent of the total number of particles and is proportional to the chosen HSS tolerance parameter $\varepsilon = 10^{-12}$, while the TTNO rank satisfies the theoretical bound, with an additional rank arising from the Laplacian-like part of the operator. The reference operator is constructed using the unstructured construction from section~\ref{subsec:unstruct_case}. The relationship between the relative error in Frobenius norm and the HSS tolerance parameter $\varepsilon$ is further investigated in Figure~\ref{fig:TTNO_unitary_alpha1_hss_tols}, where we observe that the error grows linearly with respect to the HSS tolerance. 
Finally, we vary the parameter $\alpha$ and evaluate the TTNO rank across different interaction regimes. In Figure \ref{fig:TTNO_unitary_diff_alpha_d256}, we observe that an increase in $\alpha$ results in a decrease in the representation rank.

\begin{figure}[ht]
    \centering
    \includegraphics[trim ={20mm 0mm 0 5mm},clip,scale=0.35]{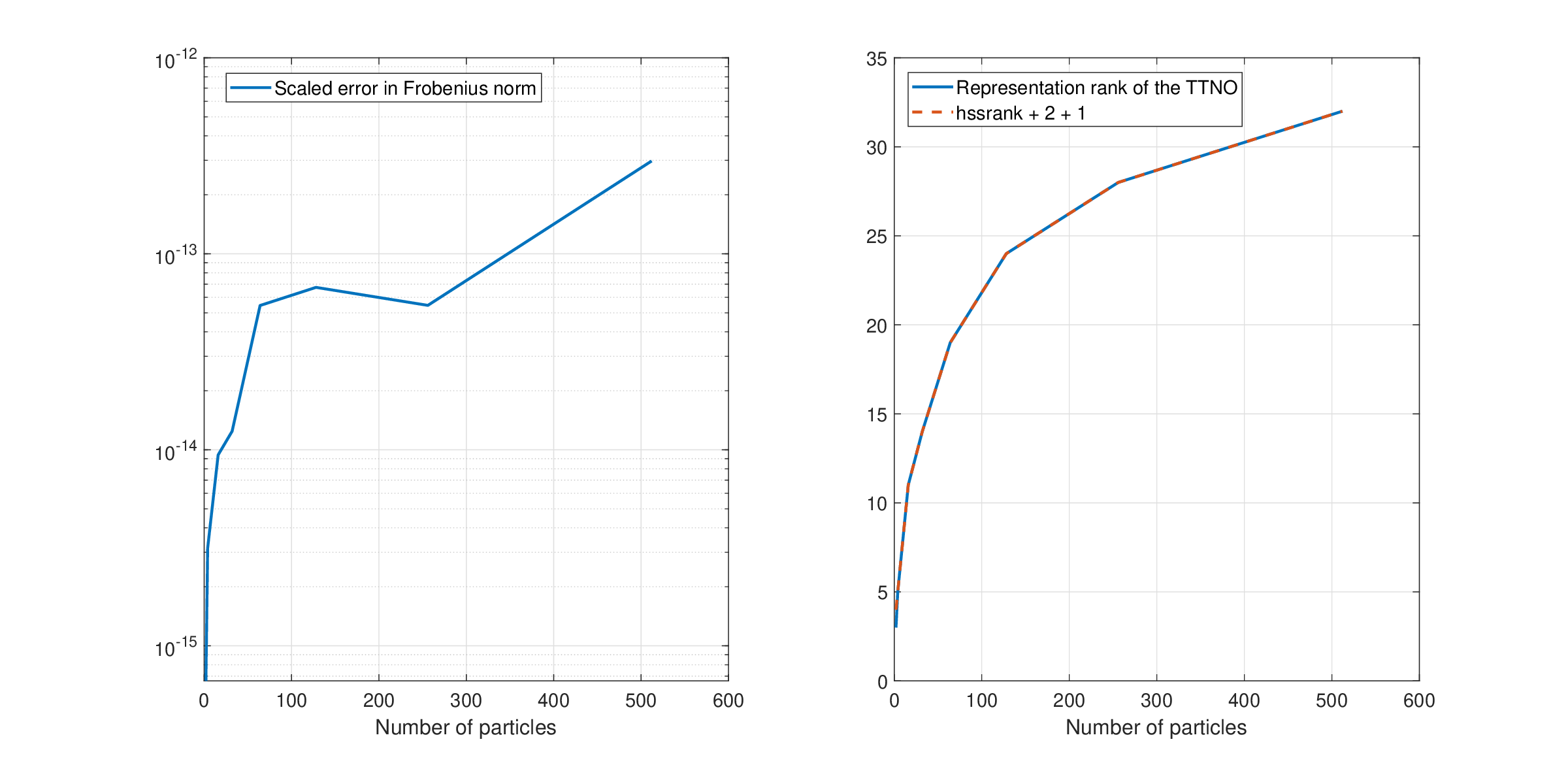}
    \caption{Long-range unitary Hamiltonian with given parameters: $\Omega=3$, $\Delta=-2$, $\nu = 2$, $\alpha = 1$ and HSS tolerance $10^{-12}$. Left: Relative error of the TTNO vs the number of particles. Right: Representation rank of the TTNO (solid line) and the excepted ranks (dashed line) versus the number of particles.}
    \label{fig:TTNO_unitary_alpha1}
\end{figure}

\begin{figure}[ht]
    \centering
    \includegraphics[trim ={40mm 0mm 0 10mm},clip,scale=0.35]{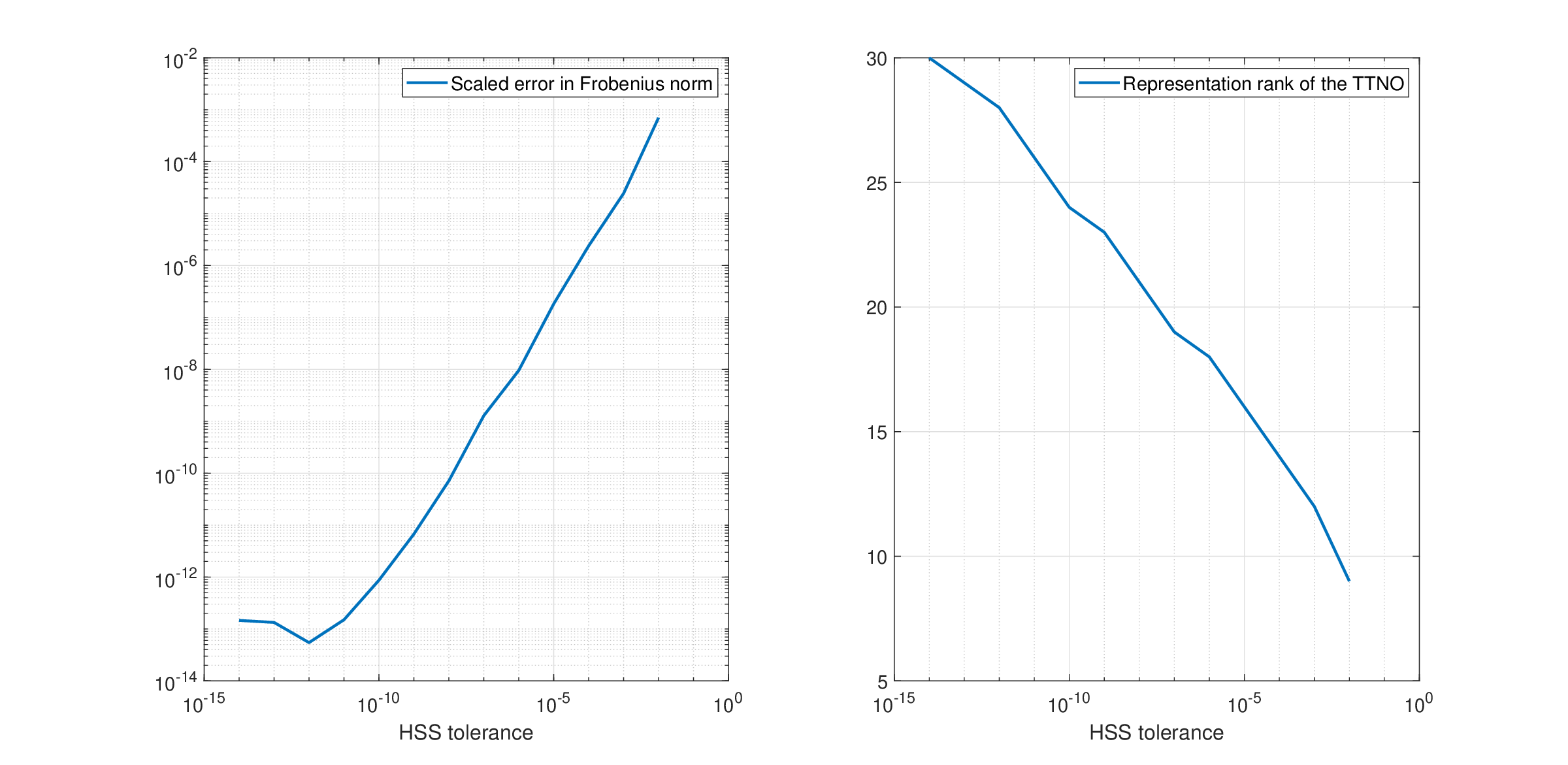}
    \caption{Long-range unitary Hamiltonian with given parameters: $\Omega=3$, $\Delta=-2$, $\nu = 2$, $\alpha = 1$ and $d=256$. Left: Relative error of the TTNO vs the HSS tolerance. Right: Representation rank of the TTNO versus the HSS tolerance.}
    \label{fig:TTNO_unitary_alpha1_hss_tols}
\end{figure}

\begin{figure}[ht]
    \centering
    \includegraphics[trim ={40mm 0mm 0 10mm},clip,scale=0.35]{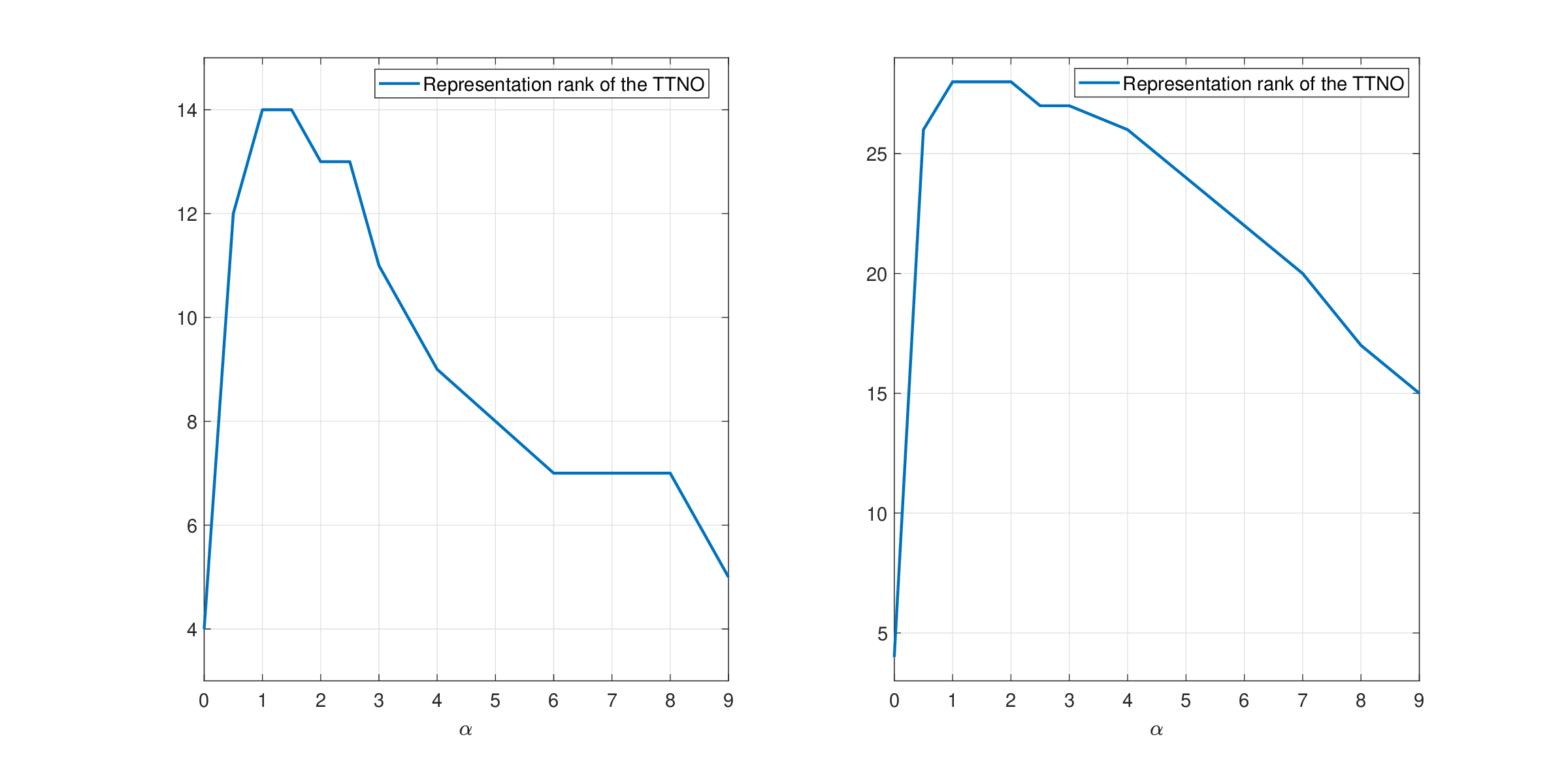}
    \caption{Long-range unitary Hamiltonian with given parameters: $\Omega=3$, $\Delta=-2$, $\nu = 2$ and $d=256$. Left: Representation rank of the TTNO versus different values of $\alpha$ computed with HSS tolerance $10^{-4}$. Right: Representation rank of the TTNO versus different values of $\alpha$ computed with HSS tolerance $10^{-12}$.}
    \label{fig:TTNO_unitary_diff_alpha_d256}
\end{figure}

\subsection{Synthetic example: Hamiltonian with interactions of large HSS rank}
In the next example, we consider the Hamiltonian given by
\begin{align*}
    \mathcal H = \Omega \sum_{k=1}^{d} \sigma_x^{(k)} + \Delta \sum_{k=1}^{d} n^{(k)} + \sum_{i< j} 
    \frac{1}{1-\text{cos}(j-i)} \cdot n^{(i)}n^{(j)} 
\end{align*}
We observe that the resulting interaction matrix admits a maximal HSS rank of $\frac{d}{2}$. Thus, the expected TTNO representation rank is $\frac{d}{2}+3$, as numerically confirmed in Figure \ref{fig:TTNO_bad_example}. This synthetic example illustrates that it is possible, but not necessarily advisable, to use TTNO representations via HSS decompositions even when a high HSS rank is expected a priori.

\begin{figure}[ht]
    \centering
    \includegraphics[trim ={40mm 0mm 0 10mm},clip,scale=0.35]{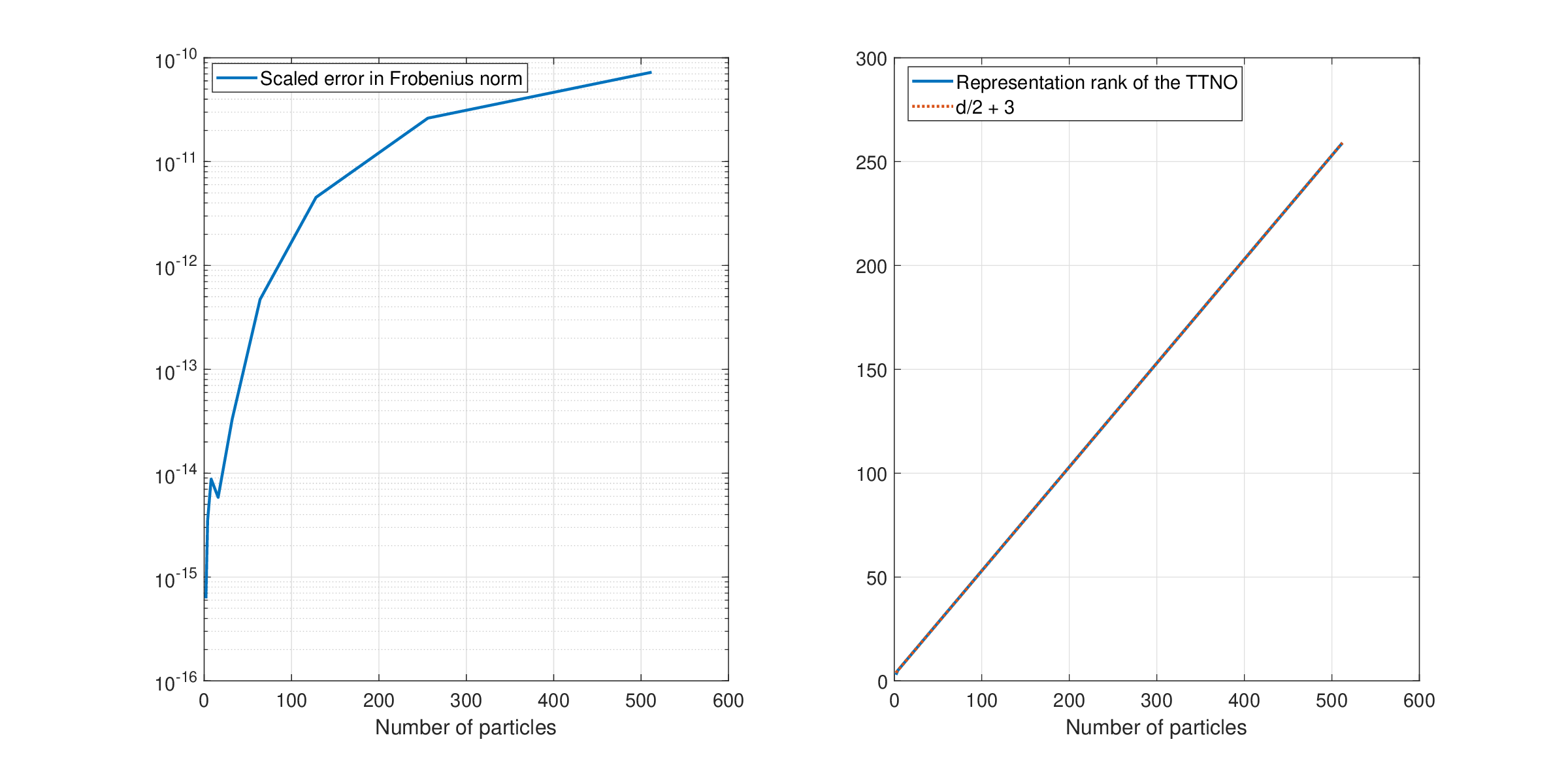}
    \caption{Long-range unitary Hamiltonian (synthetic example) with given parameters: $\Omega = 3$ and $\Delta = -2$. Left: Relative error of the TTNO vs the number of particles. Right: Representation rank of the TTNO (solid line) and $d/2 +3$ (dashed line) vs the number of particles. }
    \label{fig:TTNO_bad_example}
\end{figure}

\subsection{Open quantum spin system}\label{subsec:num_ex2}
In the next example, we consider a one-dimensional quantum systems consisting of $d$ distinguishable spin-$\frac{1}{2}$ particles following a Markovian open quantum dynamics governed by the operator
    \begin{align} \label{eq:openQuantumSystemOperator}
        \mathcal{L} &= \Omega \sum\limits_{k=1}^{d} \left[ -i \sigma_x \otimes \I + i \I \otimes \sigma_x^\top \right]^{(k)}
    	+ \Delta \sum\limits_{k=1}^{d} \left[ -\mathrm{i} n \otimes \I + \mathrm{i} \I \otimes n^\top \right]^{(k)} \\ \nonumber
    	&+ \gamma \sum\limits_{k=1}^{d} \left[ J \otimes (J^*)^\top - \frac{1}{2}J^*J \otimes \I - \frac{1}{2} \I\otimes (J^*J)^\top \right]^{(k)} \\ \nonumber
    	&+ \frac{\nu}{2c_\alpha} \sum_{i < j} \frac{-\mathrm{i}}{(j-i)^\alpha} \left[  n \otimes \I \right]^{(i)}  \left[ n \otimes \I  \right]^{(j)} + \frac{\nu}{2c_\alpha} \sum_{i < j} \frac{\mathrm{i}}{(j-i)^\alpha} \left[ \I \otimes n^\top \right]^{(i)}  \left[ \I \otimes n^\top \right]^{(j)},
    \end{align}
where
\[
    J = \begin{pmatrix} 0 & 0 \\ 1 & 0 \end{pmatrix}, \quad
    n = \begin{pmatrix} 1 & 0 \\ 0 & 0 \end{pmatrix}, \quad 
    c_\alpha = \sum_{k=1}^d \frac{1}{k^\alpha} \, .
\]
The parameters $\Omega,\Delta,\gamma,\nu,\alpha$ are model related. For a detailed description, we refer to \cite{sulz2023} and the references therein. 

We once again observe that the TTNO approximation can be achieved up to the selected HSS tolerance, independently of the number of particles, as seen in Figure \ref{fig:TTNO_open_alpha1}. Furthermore, a parameter study for $\alpha$ is illustrated in Figure \ref{fig:TTNO_open_diff_alpha_d256}, where it can be observed that the operator describing the Markovian open quantum system dynamics also admits a TTNO representation of low representation rank. The linear scaling influence of the HSS tolerance is provided in Figure~\ref{fig:TTNO_open_hss_plot}.

\begin{figure}[ht]
    \centering
    \includegraphics[trim ={40mm 0mm 0 10mm},clip,scale=0.35]{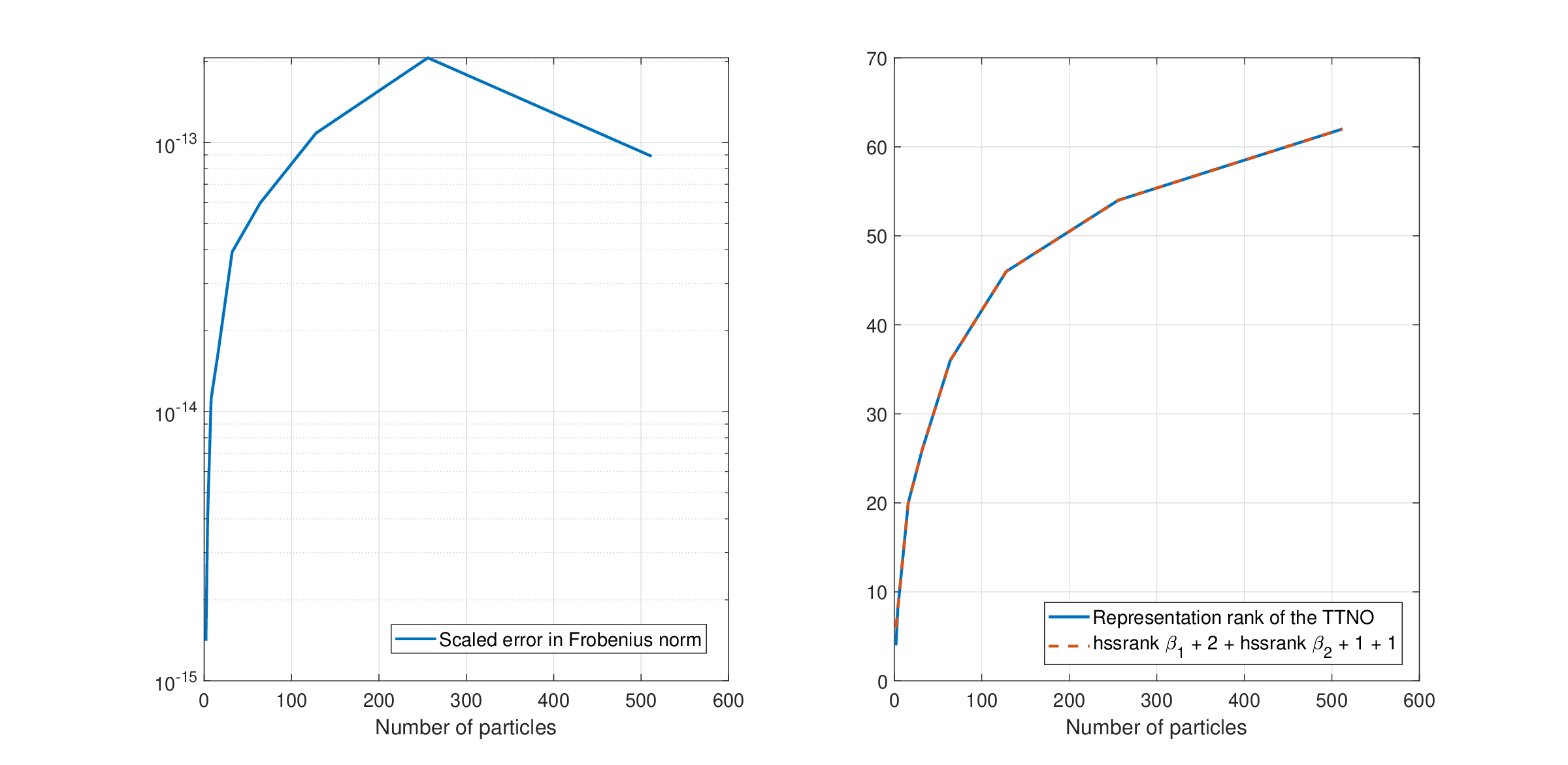}
    \caption{Open long-range Hamiltonian with given parameters: $\Omega=0.4$, $\Delta=-2$, $\gamma = 1$, $\nu = 2$, $\alpha = 1$ and HSS tolerance $10^{-12}$. Left: Relative error of the TTNO vs the number of particles. Right: Representation rank of the TTNO (solid line) and the excepted ranks (dashed line) versus the number of particles.}
    \label{fig:TTNO_open_alpha1}
\end{figure}

\begin{figure}[ht]
    \centering
    \includegraphics[trim ={40mm 0mm 0 10mm},clip,scale=0.35]{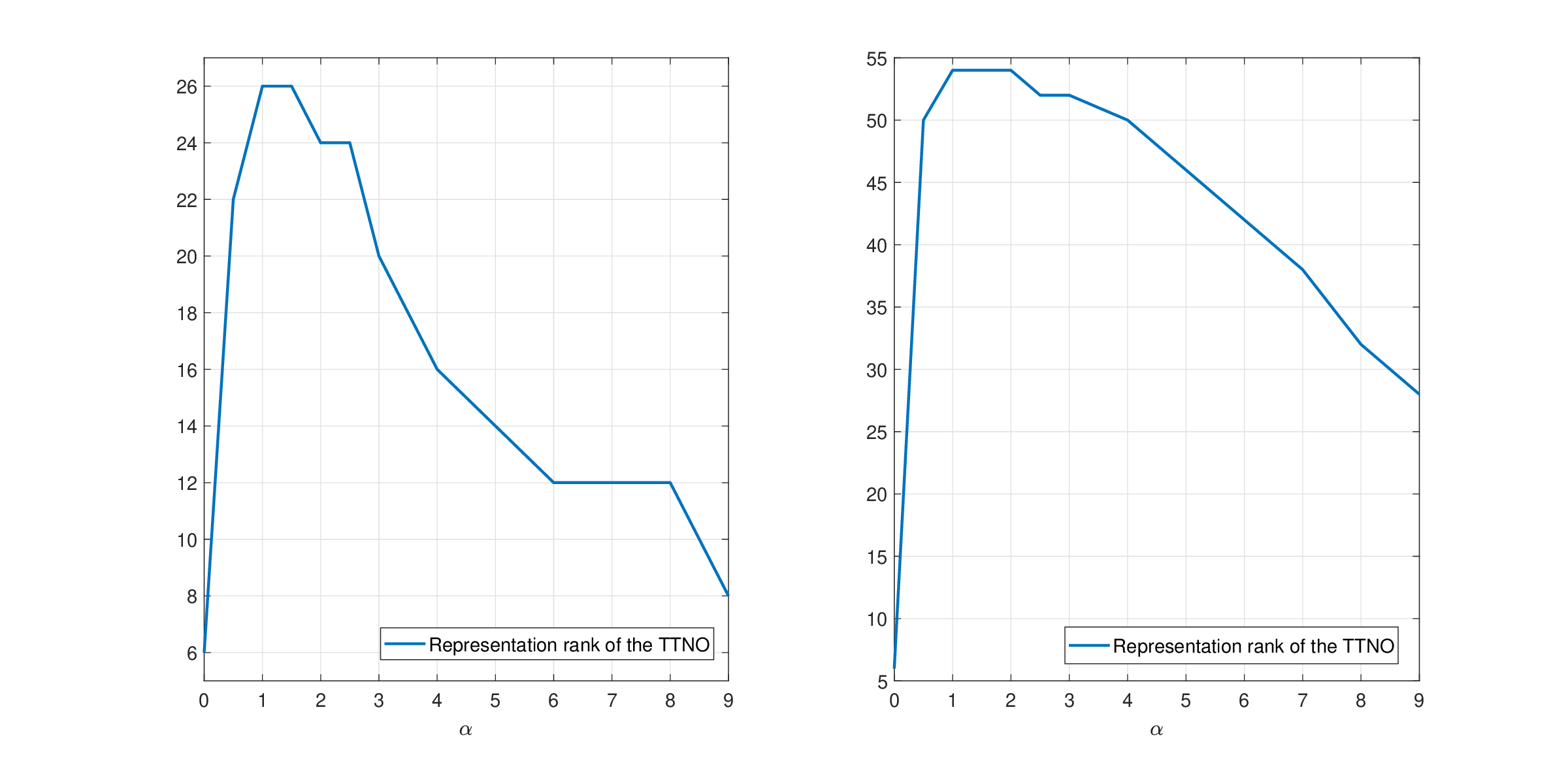}
    \caption{Open long-range Hamiltonian with given parameters: $\Omega=0.4$, $\Delta=-2$, $\gamma = 1$, $\nu = 2$, $\alpha = 1$ and $d=256$. Left: Representation rank of the TTNO versus different values of $\alpha$ computed with HSS tolerance $10^{-4}$. Right: Representation rank of the TTNO versus different values of $\alpha$ computed with HSS tolerance $10^{-12}$.}
\label{fig:TTNO_open_diff_alpha_d256}
\end{figure}

\begin{figure}[ht]
    \centering
    \includegraphics[trim ={40mm 0mm 0 10mm},clip,scale=0.35]{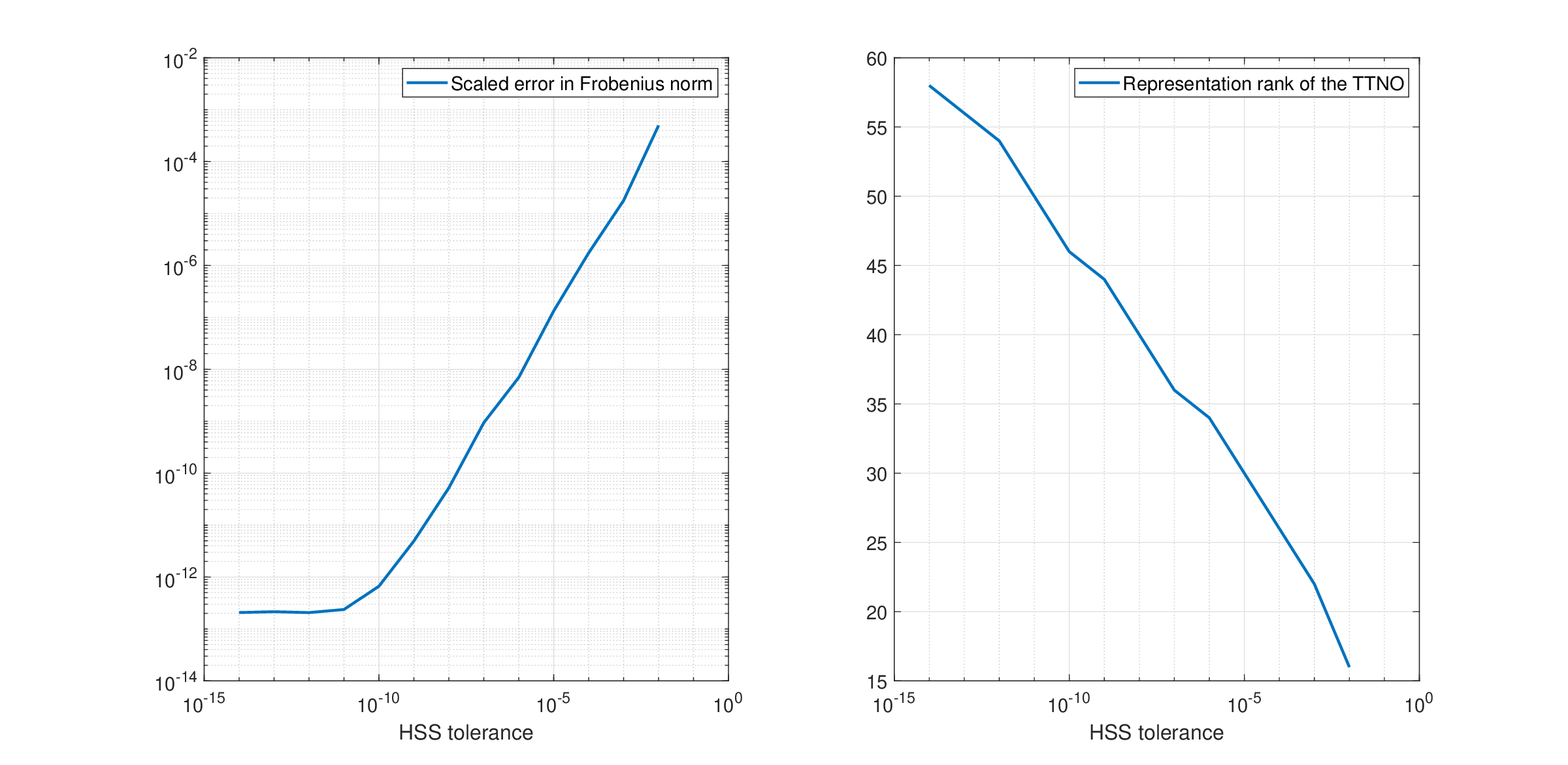}
    \caption{Open long-range Hamiltonian with given parameters: $\Omega=0.4$, $\Delta=-2$, $\gamma = 1$, $\nu = 2$, $\alpha = 1$ and $d=256$. Left: Relative error of the TTNO vs the HSS tolerance. Right: Representation rank of the TTNO versus the HSS tolerance.}
\label{fig:TTNO_open_hss_plot}
\end{figure}

\subsection{Comparison among tree tensor network formats}
To conclude, we would like to remind that the proposed methodology offers the advantage of inner flexibility, enabling the exploration of various tree tensor network formats within a single framework.

In Figure~\ref{fig:TTNO_TT_compare_alpha1}, we consider the operator \eqref{eq:closedSpinSystemOperator} and compare the ranks and memory usage of the decompositions for different numbers of particles obtained using balanced binary trees against those using unbalanced binary trees, which represent well-known tensor formats like tree tensor trains/matrix product states -- widely accepted standards in physics, including but not limited to quantum dynamics. It is observed that the representation rank produced by the unbalanced binary tree format, i.e. TT/MPS, is smaller compared to the balanced binary tree. A similar result is obtained for the open quantum system operator \eqref{eq:openQuantumSystemOperator}, as illustrated in  Figure~\ref{fig:TTNO_TT_open_compare_alpha1}. 

However, it is essential to note that while the representation rank provides valuable information, it should not be the sole reference measure for assessing TTNO compression quality. The overall storage complexity should also be taken into account. These examples showcase that the balanced binary tree format requires less memory or at most an equal amount compared to the unbalanced binary tree format (TT/MPS) as the number of particles increases. 

Furthermore, as originally suggested in \cite{CLS23} for a nearest-neighbor spin system, it's important to note that the representation rank of the \emph{state} in the TTN format tends to be significantly higher for unbalanced binary trees compared to balanced binary trees. Efficient quantum computations require careful consideration of both the representation rank of the TTNO and that of the state. This observation strongly indicates that balanced binary trees represent a promising approach for simulations involving long-range interacting quantum systems.

\begin{figure}[ht]
    \centering
    \includegraphics[trim ={40mm 0mm 0 10mm},clip,scale=0.35]{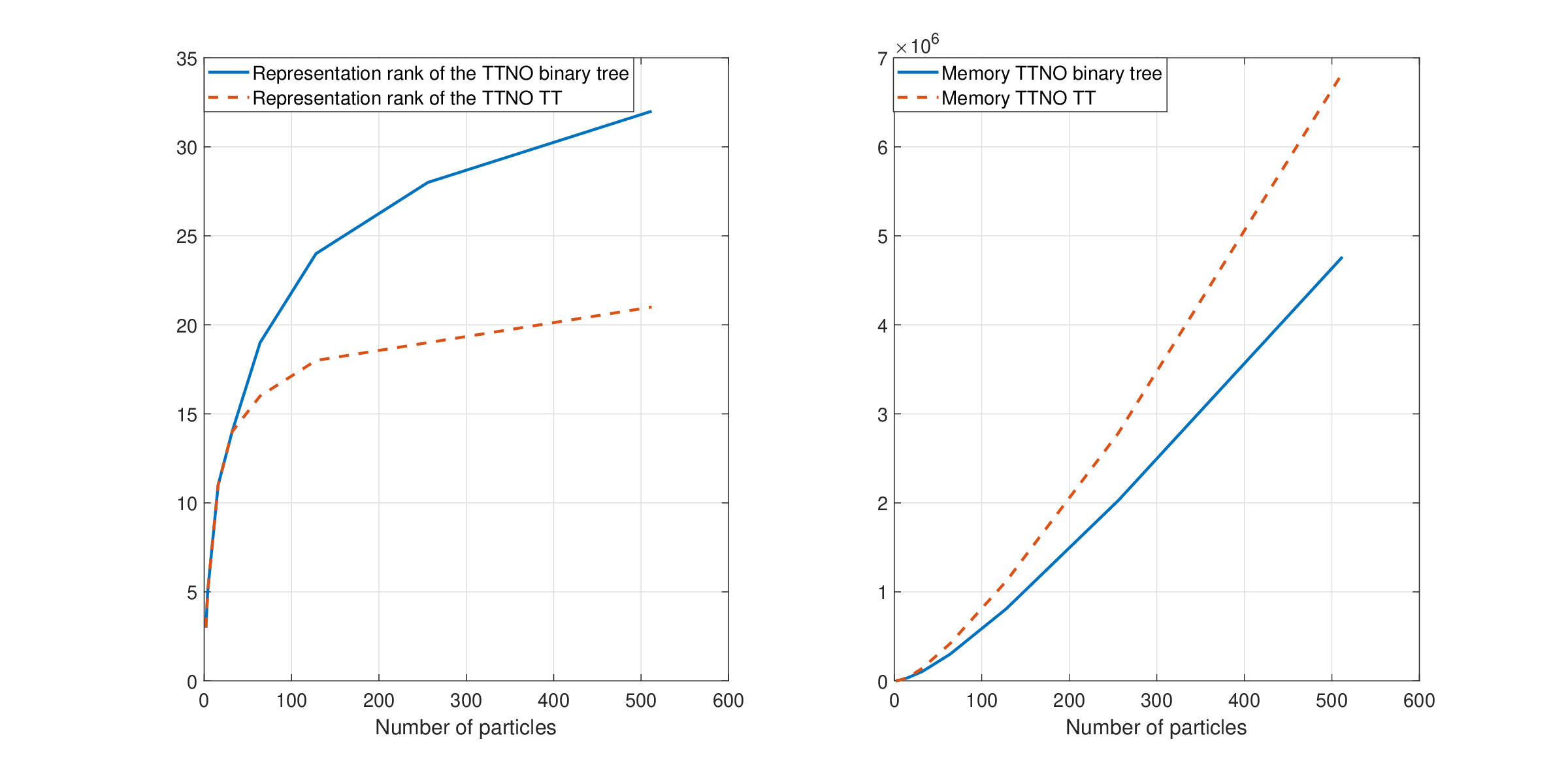}
    \caption{Long-range unitary Hamiltonian with given parameters: $\Omega=3$, $\Delta=-2$, $\gamma = 1$, $\alpha=1$ and $\nu = 2$. Left: Representation rank of the TTNO for balanced (solid line) and unbalanced binary trees (dashed line). Right: Memory in bytes of the TTNO for balanced (solid line) and unbalanced binary trees (dashed line).}
\label{fig:TTNO_TT_compare_alpha1}
\end{figure}

\begin{figure}[ht]
    \centering
    \includegraphics[trim ={40mm 0mm 0 10mm},clip,scale=0.35]{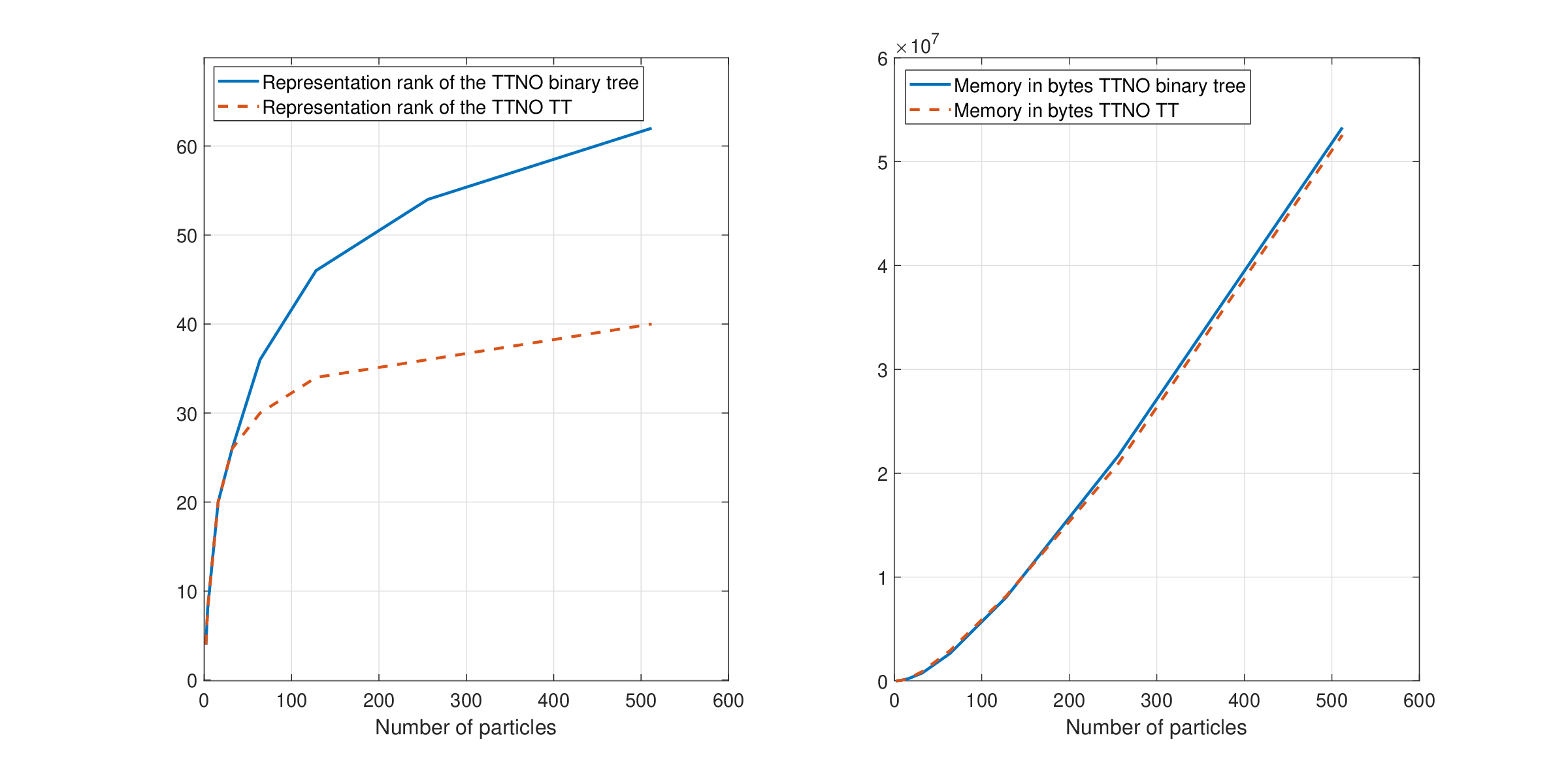}
    \caption{Open long-range Hamiltonian with given parameters: $\Omega=0.4$, $\Delta=-2$, $\gamma = 1$, $\nu = 2$, $\alpha = 1$ and HSS tolerance $10^{-12}$. Left: Representation rank of the TTNO for balanced (solid line) and unbalanced binary trees (dashed line). Right: Memory in bytes of the TTNO for balanced (solid line) and unbalanced binary trees (dashed line).}
\label{fig:TTNO_TT_open_compare_alpha1}
\end{figure}

\smallskip
\paragraph{\textbf{Acknowledgements:}}
 We thank Leonardo Robol for his help with the hm-toolbox and Christian Lubich for helpful discussions. 

\bibliographystyle{siam}
\bibliography{main}

\end{document}